\documentclass[12pt, a4paper, reqno]{amsart}

\usepackage{amsmath, amssymb, amsthm, kotex}
\usepackage[T1]{fontenc}
\usepackage[utf8]{inputenc}
\usepackage[all]{xy}
\usepackage[
  margin=1.5cm,
  includefoot,
  footskip=30pt,
  ]{geometry}
\usepackage{xcolor}
\usepackage[colorlinks=true,citecolor=cyan,linkcolor=magenta, urlcolor=violet, filecolor=green, backref=page]{hyperref}
\usepackage{indentfirst}

\newcommand{\cH}{\mathcal{H}}
\newcommand{\cL}{\mathcal{L}}

\newcommand{\cO}{\mathcal{O}}
\newcommand{\cP}{\mathcal{P}}
\newcommand{\cR}{\mathcal{R}}
\newcommand{\cS}{\mathcal{S}}

\newcommand{\bA}{\mathbb{A}}
\newcommand{\bC}{\mathbb{C}}
\newcommand{\bF}{\mathbb{F}}
\newcommand{\bG}{\mathbb{G}}

\newcommand{\bQ}{\mathbb{Q}}
\newcommand{\bR}{\mathbb{R}}

\newcommand{\bZ}{\mathbb{Z}}

\newcommand{\bfa}{\mathbf{a}}
\newcommand{\bfz}{\mathbf{z}}

\newcommand{\ff}{\mathfrak{f}}

\newcommand{\fp}{\mathfrak{p}}

\newcommand{\fs}{\mathfrak{s}}
\newcommand{\ft}{\mathfrak{t}}

\DeclareMathOperator{\Stab}{Stab}

\DeclareMathOperator{\aff}{aff}

\DeclareMathOperator{\odd}{odd}
\DeclareMathOperator{\ord}{ord}

\theoremstyle{plain}
\newtheorem{theorem}{Theorem}[section]
\newtheorem{lemma}[theorem]{Lemma}
\newtheorem{proposition}[theorem]{Proposition}
\newtheorem{corollary}[theorem]{Corollary}
\newtheorem{definition}[theorem]{Definition}

\theoremstyle{definition} 

\newtheorem{example}[theorem]{Example}

\theoremstyle{remark} 

\DeclareFontFamily{U}{wncy}{}
\DeclareFontShape{U}{wncy}{m}{n}{<->wncyr10}{}
\DeclareSymbolFont{mcy}{U}{wncy}{m}{n}
\DeclareMathSymbol{\Sha}{\mathord}{mcy}{"58}

\newcommand{\lp}{\left(}
\newcommand{\rp}{\right)}

\newcommand{\lbrb}[1]{\lp #1 \rp}
\newcommand{\lcrc}[1]{\left\{ #1 \right\}}
\newcommand{\lfrf}[1]{\left\lfloor #1 \right\rfloor}

\newcommand{\ses}[3]{ \xymatrix{
0 \ar[r] & #1 \ar[r] & #2 \ar[r] & #3 \ar[r] & 0}}

\title[Average rank of hyperelliptic curves]{Local statistics and average rank of genus $g$ hyperelliptic curves with a Weierstrass point}

\subjclass[2020]{Primary 11G30, 11G40}

\author{Keunyoung Jeong}
\address{Keunyoung Jeong: Department of Mathematics Education, Chonnam National University, 77, Yongbong-ro, Buk-gu, Gwangju 61186, Korea}
\email{keunyoung@jnu.ac.kr}

\author{Junyeong Park}
\address{Junyeong Park: Department of Mathematics Education, Chonnam National University, 77, Yongbong-ro, Buk-gu, Gwangju 61186, Korea}
\email{junyeongp@gmail.com}

\begin{document}

\begin{abstract}
We study the local reduction statistics of genus $g$ hyperelliptic curves over a number field with a Weierstrass point, ordered by naive height.
For the primes of residue characteristic greater than $2g+1$, we determined the leading term of the densities of curves whose Jacobians have prescribed abelian, toric, and unipotent ranks.
The leading terms are governed by the reduction of the defining polynomial.
Assuming the Hasse--Weil conjecture and the generalized Riemann hypothesis for the curves, we also obtain an explicit upper bound for the average analytic rank of genus $g$ hyperelliptic curves with a Weierstrass point.
\end{abstract}

\maketitle

\section{Introduction}

For certain families of elliptic curves, it is known that one can obtain an explicit upper bound for the average analytic rank under the generalized Riemann hypothesis by studying local statistics.
For example, Cho--Jeong \cite{CJ1} treated elliptic curves over $\bQ$, while Cho--Jeong \cite{CJ2} studied elliptic curves over $\bQ$ with a prescribed torsion condition.
Using the isomorphism between the compactified moduli stack of elliptic curves and the weighted projective line $\cP(4,6)$, Phillips \cite{Phi2} studied elliptic curves over a number field, and Cho--Jeong--Park \cite{CJP} considered elliptic curves over a number field with a prescribed torsion condition.
For a twist family of higher genus curves, Cho--Yoo \cite{CY} gave a result on the family $y^2 = x^{\ell} + c$ for a prime $\ell$, and the authors considered the so-called $D_{12}$-twist family of $y^2 = x^6+1$ in \cite{JP26}.

In this paper, we consider hyperelliptic curves of genus $g$ with a marked Weierstrass point. 
Let $\cH_{g, 1}$ be the moduli stack of such curves.
Then, $\cH_{g, 1}$ is isomorphic to the open substack of the weighted projective stack 
\begin{align*}
    \cP(w_g):= \cP(4, 6, \cdots, 4g+2)
\end{align*}
(cf. \cite{HP23}, Section \ref{subsec: reduction}).
Given a ring $R$, by abuse of notation, we denote by $\cP(w_g)(R)$ the set of isomorphism classes of the groupoid of $R$-points of the stack $\cP(w_g)$, following \cite{BN, CJP}.

Let $K$ be a number field and let $\fp$ be a prime of $K$.
For a given hyperelliptic curve $C$, we denote by $a(C), t(C),$ and $u(C)$ its abelian, toric, and unipotent rank of $C$ at $\fp$ (cf. section \ref{subsec: naive red}).
We will also define the height $H_{w_g, K}$ on $\cP(w_g)(K)$ in section \ref{subsec: reduction}.
For a tuple of nonnegative integers $\tau = (a, t, u)$ with $a + t + u= g$, we define the probability that a hyperelliptic curve $C$ over $K$ of genus $g$ satisfies $\tau$ at $\fp$ by
\begin{align} \label{eqn: property}
\mathrm{Pr}_{\fp}(a, t, u):= 
    \lim_{X \to \infty} \frac{\# \lcrc{C \in \cH(g, 1)(K) : H_{w_g, K}(C) \leq X, (a(C/K_{\fp}), t(C/K_{\fp}), u(C/K_{\fp}) = \tau } }{ \# \lcrc{ C \in \cH_{g, 1}(K) : H_{w_g, K}(C) \leq X}  }
\end{align}
if it exists.
On the other hand, for a partition $\lambda \vdash 2g+1$, we define
\begin{align*}
    a_0(\lambda) := \frac{\# \lcrc{\textrm{odd parts of $\lambda$}} - 1}{2}, \qquad 
    t_0(\lambda) := \# \lcrc{\textrm{even parts of $\lambda$}}
\end{align*}
and $u_0(\lambda) := g - a_0(\lambda) - t_0(\lambda)$.
Here, the parts are counted with multiplicity.
(For the geometric meaning of the invariants, see Proposition \ref{prop: naive atu}.)
Also, we define
\begin{align*}
    M_g(a, t, u) := \# \lcrc{\lambda \vdash 2g+1 : (a_0(\lambda), t_0(\lambda), u_0(\lambda)) = (a, t, u) }.
\end{align*}

\begin{theorem} \label{mainthm: probability}
Let $K$ be a number field and let $\fp$ be a prime of $K$ with residue field  $\kappa(\fp) \cong \bF_q$ of residue characteristic $> 2g+1$.
Then, for every triple of nonnegative integers $(a, t, u)$ with $a+t+u = g$,
\begin{align*}
    \mathrm{Pr}_{\fp}(a, t, u) = \frac{M_g(a, t, u)}{q^{t + 2u}} + O_g(q^{-t-2u-1}).
\end{align*}
Moreover, 
\begin{align*}
    \mathrm{Pr}_{\fp}(g, 0, 0) = 1 - \frac{1}{q} + O_g(q^{-2}).
\end{align*}
\end{theorem}
The constant $M_g(a, t, u)$ can be effectively computed.
For example, we have
\begin{align*}
    \mathrm{Pr}_{\fp}(g, 0, 0) = 1 - \frac{1}{q} + O_g(q^{-2}), \quad
    \mathrm{Pr}_{\fp}(g-1, 1, 0) = \frac{1}{q} + O_g(q^{-2}), \quad
    \mathrm{Pr}_{\fp}(g-1, 0, 1) = \frac{1}{q^2} + O_g(q^{-3}).
\end{align*}
When $g = 1$, it recovers the leading asymptotics of \cite[Theorem 1.4]{CJ1} and \cite[Theorem 1.1.2]{Phi2}, which give the probabilities of good,  multiplicative, and additive reduction at $p > 3$ as
\begin{align*}
    \frac{q^2-q}{q^2}\frac{1}{1 - q^{-10}}, \qquad 
    \frac{q-1}{q^2} \frac{1}{1 - q^{-10}}, \qquad 
    \frac{1}{q^2} + O(q^{-10}),
\end{align*}
respectively.

Another main theorem of this paper is to establish an upper bound on the average analytic rank.

\begin{theorem} \label{mainthm: rank}
Let $K$ be a number field of degree $D$.
Assume the Hasse--Weil conjecture and the generalized Riemann hypothesis for the $L$-functions of hyperelliptic curves over $K$ of genus $g$ with a Weierstrass point.
Then, the limsup of the average analytic rank of genus $g$ hyperelliptic curves with a Weierstrass point is bounded above by
$\frac{1}{2} + \frac{3g(2g+1)D}{2}$.
\end{theorem}
When $g = 1$, we recover the bound $\frac{1+9D}{2}$ obtained by Phillips \cite[Theorem 1.1.1]{Phi2}.
In light of the Katz--Sarnak philosophy, we expect that the average analytic rank of genus $g$ hyperelliptic curves with a Weierstrass point is $\frac{1}{2}$.

On the other hand, Bhargava--Gross \cite{BG13} showed that the average algebraic rank of the Jacobians of genus $g$ hyperelliptic curves over $\bQ$ with a rational Weierstrass point is bounded above by $\frac{3}{2}$.
This is a stronger bound than ours, but two points are worth emphasizing.
First, our result applies to hyperelliptic curves over an arbitrary number field.
Second, the average analytic rank is of independent interest, because the Birch--Swinnerton-Dyer conjecture is required to equate the algebraic rank and the analytic rank.

\smallskip 

The main idea is to introduce a naive model and its reduction modulo $\fp$ for a hyperelliptic curve with a Weierstrass point. 
Let $F$ be a nonarchimedean local field with the ring of integers $\cO_F$, uniformizer $\pi_F$, and residue field $k$ whose characteristic is $p > 2g+1$.
Such a curve $C$ over $F$ admits a monic Weierstrass equation of the form $y^2 = f(x)$ where
\begin{align*}
    f(x) = x^{2g+1} + a_2 x^{2g-1} + \cdots + a_{2g+1}, \qquad a_i \in \cO_F.
\end{align*}
We choose $a_i$ to be a weighted primitive integral representative, unique up to $w_g$-weighted $\cO_F^\times$-action (cf. section \ref{subsec: reduction}). 
Let $\mathcal{C}$ be the associated naive model of $C$, whose special fiber is a projective completion of $y^2 = \overline{f}(x)$ where
\begin{align*}
    \overline{f}(x) = x^{2g+1} + \overline{a_2} x^{2g-1} + \cdots + \overline{a_{2g+1}} \in k[x].
\end{align*}
Here, $\overline{a_i}$ is a reduction of $a_i$ modulo $\pi_F\cO_F$.
The factorization of $\overline{f}(x)$ in the algebraic closure of $k$ determines the abelian, toric, and unipotent rank of the Jacobian of the naive special fiber.
We denotes these ranks by $a_0(C), t_0(C),$ and $u_0(C)$, respectively (cf. Proposition \ref{prop: naive atu}).

The triple $\tau(C):= (a(C), t(C), u(C))$ using the N\'eron model of Jacobian need not coincide with $\tau_0(C):= (a_0(C), t_0(C), u_0(C))$ (cf. Example \ref{example: tau tau0}).
Nevertheless, considering $\tau_0(C)$ has several advantages from the point of view of arithmetic statistics.
First, the weighted projective parametrization of hyperelliptic curves with a Weierstrass point allows us to count curves with prescribed naive ranks by adapting the counting method of \cite{CJP}. 
In particular, the coefficientwise reduction of the chosen integral equation is compatible with the reduction procedure of \cite{CJP}. 
Second, the two rank triples coincide for all but a small proportion of curves. 
Consequently, the discrepancy between $\tau(C)$ and $\tau_0(C)$ can be absorbed into the error terms of Theorem \ref{mainthm: probability}. 
Moreover, the naive reduction also provides a way to compute the local Frobenius traces for curves whose reduced defining polynomial has only simple or double roots. 
Together with estimates for the remaining singular strata, this enables us to adapt the strategy of \cite{CJP} to prove Theorem \ref{mainthm: rank}.

We briefly describe the main ingredients underlying the comparison between $\tau(C)$ and $\tau_0(C)$.
Proposition \ref{prop:regular-lifts} shows that the two triples coincide whenever the naive model $\mathcal C$ is regular, while Proposition \ref{prop: semistable simple double reduction} establishes the same conclusion when every root of $\overline f$ has multiplicity at most $2$.
Also, Proposition \ref{lem:rank-comparison} shows that $a(C) \geq a_0(C)$ and $t(C) \geq t_0(C)$
In the former case, the equality follows from the classical description of the identity component of the special fiber of the Néron model as the generalized Jacobian $\mathrm{Pic}^0(\mathcal C_k)$ of the special fiber of $\mathcal C$. 
The latter cases require a more delicate analysis and rely crucially on the theory of cluster pictures, developed by Dokchitser--Dokchitser--Maistret--Morgan \cite{DDMM} and systematically presented in the user's guide \cite{BBB}.

The paper is organized as follows:
In section \ref{sec: setup}, we introduce the naive model, its reduction, and give some rank comparison results.
Together with counting results on finite fields, we prove Theorem \ref{mainthm: probability} in section \ref{sec: proof main prob}.
Section \ref{sec: AR} is devoted to the proof of Theorem \ref{mainthm: rank}.

\smallskip 

\textbf{Acknowledgement.} 
The authors thank Raymond van Bommel for pointing out the error in the first draft.
K. Jeong was supported by the National Research Foundation of Korea (NRF) grant funded by the Korea government(MIST) (RS-2024-00341372 and RS-2024-00415601).
J. Park was supported by Basic Science Research Program through the National Research Foundation of Korea(NRF) funded by the Ministry of Education (RS-2024-00449679) and the National Research Foundation of Korea(NRF) grant funded by the Korea government (MSIT) (RS-2024-00415601).

\smallskip 

\textbf{AI disclosure} 
The first author used ChatGPT Sol and Astra during the preparation of this paper.
The idea that the strategy of \cite{CJP} could be applied to the set of hyperelliptic curves with a Weierstrass point emerged during the discussions with AI, and the counting lemmas in section \ref{subsec: FF counting} were largely generated by models.
Interestingly, the models did not identify the error pointed out by R. van Bommel in the first draft, but later suggested the notion of naive model, Proposition \ref{prop:regular-lifts}, and Proposition \ref{prop: semistable simple double reduction}.
The authors take full responsibility for the mathematical content of the paper.

\section{Setup and rank comparisons}  \label{sec: setup}

\subsection{Naive model and N\'eron model}  \label{subsec: naive red}

Let $F$ be a non-archimedean local field with ring of integers $R$, uniformizer $\pi$, and residue field $k$ of cardinality $q$ and characteristic $p>2g+1$.
We use the normalized valuation $v$ with $v(\pi)=1$.

For a monic separable polynomial $f\in R[x]$ of degree $2g+1$, let $C/F$ be the smooth projective curve associated to $y^2=f(x)$. 
Let $J$ be the Jacobian of $C$ and let $\mathcal{J}$ be the N\'eron model of $J$ over $R$.
Then the identity component $\mathcal{J}_k^0$ of the special fiber is a smooth connected commutative algebraic group over $k$.
By Chevalley's theorem, there is an exact sequence
\begin{align*}
\ses{L}{\mathcal{J}_k^0}{B}
\end{align*}
where $L$ is a connected affine algebraic group and $B$ is an abelian variety over $k$.
Since every connected commutative affine algebraic group over a perfect field is isomorphic to a product of a unipotent group $U$ and a torus $T$, we obtain
\begin{align*}
    \xymatrix{
    0 \ar[r] & U \times T \ar[r] & \mathcal{J}_k^0 \ar[r] & B \ar[r] & 0 .
    }
\end{align*}
We define the abelian rank, toric rank, and unipotent rank of $C$ by
\begin{align*}
    a(C):=\dim B, \qquad t(C):=\dim T, \qquad u(C):=\dim U.
\end{align*}
For details, see \cite[Sections 7.5 and 10.1]{Liu02}.
Since $\dim \mathcal{J}_k^0=\dim J=g$, we have
\begin{align*}
    a(C) + t(C) + u(C) =g.
\end{align*}
We also use $a(f), t(f), u(f)$ instead of $a(C), t(C), u(C)$, and define
\begin{align*}
    \tau(f):=(a(f),t(f),u(f)).
\end{align*}

Next, we will define a naive reduction of $C$ associated to $f$.
Let $\mathcal C/R$ be the integral hyperelliptic model obtained by gluing
\begin{align*}
    y^2=f(x),
    \qquad
    w^2=z^{(2g+1)+1}f(z^{-1}),
\end{align*}
with $z=x^{-1}$ and $w=yx^{-(g+1)}$.
We call its special fiber $\mathcal{C}_k:= \mathcal{C} \otimes_R k$ the naive reduction associated to $f$.
We note that $\mathcal{C}_{k}$ is obtained by gluing
\begin{align*}
    y^2 = \overline{f}(x), \qquad w^2 = z^{(2g+1)+1}\overline{f}(z^{-1})
\end{align*}
and hence is the natural projective completion of the affine curve $y^2 = \overline{f}(x)$.
Since $\overline f$ has odd degree, it is not a square in $\overline k[x]$, and hence $\mathcal C_k$ is geometrically integral. 
Moreover, $\mathcal C$ is smooth over $R$ at its point at infinity, since $z^{2g+2}f(z^{-1})=z+O(z^2).$
Let 
\begin{align*}
    \overline{f}(x) = \prod_{i=1}^k (x - \alpha_i)^{m_i}, \qquad \alpha_i \in \overline{k}, 
\end{align*}
where $\alpha_i \in \overline{k}$ are distinct and $m_1 \geq m_2 \geq \ldots \geq m_k$.
We define the geometric multiplicity partition of $\overline{f}$ to be
\begin{align*}
    (m_1, \ldots, m_k) \vdash 2g+1.
\end{align*}

\begin{definition}
\label{def:partition-ranks}
For a partition $\lambda=(m_1,\ldots,m_s)\vdash 2g+1$, let $r(\lambda)$ denote the number of parts of the partition $\lambda$. 
Define
\begin{align*}
    &a_0(\lambda)
    :=g-\sum_{j=1}^s
        \left\lfloor\frac{m_j}{2}\right\rfloor, \quad 
    t_0(\lambda)
    :=\#\{j:2\mid m_j\},\quad 
    u_0(\lambda)
    :=\sum_{j=1}^s
        \left\lfloor\frac{m_j-1}{2}\right\rfloor, 
\end{align*}
and 
\begin{align*}
\tau_0(\lambda)
    :=(a_0(\lambda),t_0(\lambda),u_0(\lambda)), \qquad 
    d(\lambda)
    :=\sum_{j=1}^s(m_j-1)=(2g+1)-s.
\end{align*}
\end{definition}

If the geometric multiplicity partition of $\overline{f}$ is $\lambda$, then we define
\begin{align*}
    a_0(f) := a_0(\lambda), \qquad 
    t_0(f) := t_0(\lambda), \qquad 
    u_0(f) := u_0(\lambda).
\end{align*}
Since $\lfloor \frac{m}{2} \rfloor = \frac{1+(-1)^m}{2} + \lfloor \frac{m-1}{2} \rfloor$, we have $a_0(\lambda) + t_0(\lambda) + u_0(\lambda) = g$ and similarly $d(\lambda) = t_0(\lambda) + 2u_0(\lambda)$.


\begin{proposition}
\label{prop: naive atu}
Let $\lambda=(m_1,\ldots,m_s) \vdash 2g+1$ be the geometric multiplicity partition of $\overline f$. 
Then the abelian, toric, and unipotent ranks of the generalized Jacobian $\operatorname{Pic}^0_{\mathcal C_k/k}$ are $a_0(\lambda)$, $t_0(\lambda)$, and $u_0(\lambda)$, respectively.
\end{proposition}
\begin{proof}
We may work over $\overline k$. 
Let $\widetilde{\mathcal{C}_{\overline{k}}}$ be the normalization of $\mathcal{C}_{\overline{k}}$, and let
\begin{align*}
    \overline{f}(x) = \prod_{i=1}^k (x - \alpha_i)^{m_i}, \qquad 
    Q(x) = \prod_{i=1}^k (x - \alpha_i)^{\lfloor \frac{m_i}{2} \rfloor}, \qquad 
    h(x) = \prod_{m_i, \textrm{odd}} (x- \alpha_i)
\end{align*}
so that $\overline{f}(x) = Q(x)^2h(x)$.
Then, $\widetilde{\mathcal{C}_{\overline{k}}}$ is the smooth projective model of $y^2 = h(x)$.
Since $h$ is squarefree of odd degree, the genus of $\widetilde{\mathcal{C}_{\overline{k}}}$ is
\begin{align*}
\frac{\deg h - 1}{2} = 
g-\sum_j\left\lfloor\frac{m_j}{2}\right\rfloor
=
a_0(\lambda).
\end{align*}
By the standard description of the Picard scheme of a singular curve \cite[p.951]{JRS}, the normalization induces an exact sequence
\begin{align*}
0\longrightarrow L
\longrightarrow \operatorname{Pic}^0(\mathcal{C}_{\overline{k}})
\longrightarrow \operatorname{Jac}(\widetilde{\mathcal{C}_{\overline{k}}})
\longrightarrow0,
\end{align*}
where $L$ is a connected affine group.
Thus the abelian rank of $\mathrm{Pic}^0(\mathcal{C}_{\overline{k}})$ is $a_0(\lambda)$.

At a root of multiplicity $m_j\geq2$, the completed local equation is $Y^2=X^{m_j}$.
It has two branches when $m_j$ is even, and one branch when $m_j$ is odd.
Since $\mathcal{C}_{\overline{k}}$ is integral, the description of the toric part in \cite[Section 1.2, (1.2.1)]{JRS} gives
\begin{align*}
t=\sum_{j=1}^s(b_j-1)
=\#\{j:2\mid m_j\}
=t_0(\lambda),
\end{align*}
where $b_j$ is the number of branches at the corresponding point.
Finally, since the arithmetic genus of $\mathcal{C}_{\overline{k}}$ is $g$,
the unipotent rank is
\begin{align*}
u
&=g-a_0(\lambda)-t_0(\lambda)
=\sum_{j=1}^s
\left\lfloor\frac{m_j-1}{2}\right\rfloor
=u_0(\lambda).
\end{align*}
This completes the proof.
\end{proof}

\begin{proposition} \label{prop:regular-lifts}
Let 
\begin{align*}
    \overline{f}(x) = \prod_{i=1}^r h_i(x)^{e_i}
\end{align*}
be the factorization of $\overline{f}(x)$ into distinct monic irreducible polynomials in $k[x]$.
For $i$ satisfying $e_i \geq 2$, choose a root $\overline{c}_i \in \overline{k}$ of $h_i(x)$ and a lifting $c_i \in \cO_{F_i}$ where $F_i/F$ is unramified extension with residue field $k(\overline{c}_i)$.
Then, 
\begin{enumerate}
    \item $\mathcal{C}$ is regular if and only if $v(f(c_i)) = 1$ for $i$ satisfying $e_i \geq 2$.
    \item When $\mathcal{C}$ is regular, we have $\tau(f) = \tau_0(\lambda)$, where $\lambda$ is the multiplicity partition of $\overline{f}$.
\end{enumerate}
\end{proposition}

\begin{proof}
Regularity may be checked after unramified extension \cite[\href{https://stacks.math.columbia.edu/tag/06QN}{Tag 06QN}]{Stacks}.
At a repeated root, i.e., $e_i\geq2$, write $x=c_i+X$ and
\begin{align*}
f(c_i+X)
&=\widetilde a_0+\widetilde a_1X+\cdots
  +\widetilde a_{e_i}X^{e_i}+\cdots,
\end{align*}
where $\widetilde a_0,\ldots,\widetilde a_{e_i-1} \in\pi\mathcal O_{F_i}$ and $\widetilde a_{e_i}\in\mathcal O_{F_i}^\times$.
Since $e_i\geq2$, we have
\begin{align*}
Y^2-f(c_i+X)
\equiv-\widetilde a_0
\pmod{(\pi,X,Y)^2}.
\end{align*}
Thus the local ring of $\mathcal C$ at this point is regular if and only if $v(f(c_i))=1$.

At all other points of the special fiber, the model is smooth over $R$, as noted above for the point at infinity.
Since the generic fiber is smooth, this proves the regularity criterion.
The condition is independent of the chosen lift $c_i$, since $f'(c_i)\in\pi\mathcal O_{F_i}$ and hence
\begin{align*}
f(c_i+\pi\xi)\equiv f(c_i)\pmod{\pi^2}.
\end{align*}

Now suppose that $\mathcal C$ is regular.
Since $\mathcal C_k$ is geometrically integral and reduced, its unique irreducible component has multiplicity one.
By Raynaud's description of the N\'eron model, as recalled in \cite[Section 1.3]{JRS}, we obtain $\mathcal J_k^0
\cong \operatorname{Pic}^0_{\mathcal C_k/k}.$
Therefore, Proposition \ref{prop: naive atu} gives $\tau(f)=\tau_0(\lambda)$.
\end{proof}

\subsection{Cluster picture and its applications}

We first recall the cluster picture introduced by \cite{DDMM, BBB}.
Let $F$ be a local field with a valuation $v$.
For a polynomial $f$ over $F$, let $\cR$ be the set of roots of $f$ in $\overline{F}$.
We say that $\fs \subset \cR$ is a \emph{cluster} if 
\begin{align*}
    \fs = \cR \cap \lcrc{x \in \overline{\bQ_p} : v(x - z) \geq d}
\end{align*}
for some $z \in \overline{F}$ and $d \in \bQ$.
When $\fs$ is a cluster of size $> 1$, we say that $\fs$ is \emph{proper}.
For a cluster $\fs$ of size $\geq 2$, we define the \emph{depth} by
\begin{align*}
    d_{\fs} := \min_{r, r' \in \fs} v(r - r'),
\end{align*}
and an element $z_{\fs} \in \overline{F}$ is called \emph{centre} if 
\begin{align*}
    \min_{r \in \fs} v(z_{\fs} - r) = d_{\fs}.
\end{align*}
If $\fs'$ is a maximal subcluster of $\fs$, we say that $\fs'$ is a \emph{child} of $\fs$. 
For a cluster $\mathfrak{s}$, we denote by $P(\mathfrak{s})$ the cluster such that $\mathfrak{s}$ is its child, and call it the \emph{parent} of $\mathfrak{s}$.
If $\fs$ is a cluster of even cardinality whose children are all even, then we say that $\fs$ is \emph{\"ubereven}.
For a cluster $\fs, \fs'$, we define $\fs \wedge \fs'$ as the smallest cluster containing $\fs$ and $\fs'$.
A cluster is \emph{twin} if its size is $2$, and \emph{cotwin} if it is non-\"ubereven and has a child of size $2g$, where $g$ is the genus of the curve $y^2 = f(x)$.
For a cluster $\fs$,
\begin{align*}
    \fs^*:= \begin{cases}
        \textrm{smallest $\fs^* \supset \fs$ that does not have an \"ubereven parents, } & \textrm{if $\fs$ is not cotwin,} \\
        \textrm{the child of $\fs$ of size $2g$, } & \textrm{if $\fs$ is cotwin,}
    \end{cases}
\end{align*}
A cluster $\fs$ is \emph{principal} if $|\fs| < 2g + 2$, proper, not a twin or a cotwin, or if $|\fs| = 2g + 2$, $\fs$ has $\geq 3$ children, and is not a cotwin.
Let $c$ be the leading coefficient of $f$.
For a proper cluster $\fs$, we define
\begin{align*}
    \nu_{\fs} &:= v(c) + |\fs| d_{\fs} + \sum_{r \not\in \fs} d_{\lcrc{r} \wedge \fs}, \qquad 
    \tilde{\lambda}_{\fs} := \frac{1}{2} \lbrb{ v(c) + |\widetilde{\fs}| d_{\fs} + \sum_{r \not\in \fs} d_{\lcrc{r} \wedge \fs}}.
\end{align*}

\begin{theorem} \label{thm: BBB S 5}
\begin{enumerate}
    \item The curve $C$ is semistable if and only if the following conditions are satisfied:
    \begin{itemize}
        \item $F(\cR)/F$ has ramification degree at most $2$.
        \item Every proper cluster is invariant under the action of the inertia group $I_F$.
        \item Every principal cluster $\fs$ has $d_{\fs} \in \bZ$ and $\nu_{\fs} \in 2\bZ$.
    \end{itemize}
    \item The potential toric rank of the Jacobian of $C$ is equal to the number of even non-\"ubereven clusters $\fs \neq \cR$, less $1$ if $\cR$ is \"ubereven.
\end{enumerate}
\end{theorem}
\begin{proof}
This is \cite[Theorem 5.1, Theorem 5.6]{BBB}, respectively.
\end{proof}

\begin{example} \label{example: tau tau0}
Let $p>5$ and let $C/\bQ_p$ be a hyperelliptic curve defined by the equation
\begin{align*}
    y^2 = (x^3 - p^6)(x^2 - 1).
\end{align*}
Since $\overline{f}(x) = x^3(x^2 - 1)$, we have $\tau_0(C) =  (1, 0, 1)$ by Definition \ref{def:partition-ranks}.
On the other hand, its roots are $\pm 1, p^2, p^2\zeta_3,$ and $ p^2 \zeta_3^2$.
Hence, every root is defined over an unramified extension, and the cluster picture is 
$\lcrc{ \bullet, \bullet, \lcrc{\bullet, \bullet, \bullet}_2}_0$.
The only proper clusters are $\cR$ and $\fs := \lcrc{p^2, p^2\zeta_3, p^2\zeta_3^2}$ with $d_{\cR} = 0$ and $d_{\fs} = 2$.
We also have $\nu_{\cR} = 0$ and $\nu_{\fs} = 6$.
Therefore, $C$ is semistable by Theorem \ref{thm: BBB S 5} (1) so that $u(C) = 0$.
Also, Theorem \ref{thm: BBB S 5} (2) shows that the potential toric rank is $0$, so we have $t(C) = 0$ since $C$ is semistable.
Therefore, $\tau(C) = (2, 0, 0)$.
\end{example}

\begin{proposition}\label{prop: semistable simple double reduction}
Let $C$ be a hyperelliptic curve $y^2 = f(x)$ for a monic $f \in \cO_F[x]$ of degree $2g+1$.
Suppose that every root of $\overline f$ over $\overline k$ has multiplicity at most $2$, and let $r$ be the number of double roots of $\overline f$. 
Let $\overline{C}/k : y^2 = \overline{f}(x)$ denote the corresponding projective curve over $k$. Then:
\begin{enumerate}
\item The curve $C$ and its Jacobian $J$, has semistable reduction over $K$.

\item In this case, we have $(a(f), t(f), u(f)) = (a_0(f), t_0(f), u_0(f))$.

\item  For every $m\geq 1$, the Frobenius trace appearing in the local Euler factor of $C/F$ is computed by the naive special fiber.
Hence, $a_{\mathfrak p^m}(C)=q_{\mathfrak p}^m+1-\#\overline C(\mathbb F_{q_{\mathfrak p}^m}).$
\end{enumerate}
\end{proposition}
\begin{proof}
Since $f$ is monic integral, its roots are integral. 
Since the multiplicity of $\overline{f}$ is at most $2$, $\overline{f}$ has at least two elements, which means that $d_{\mathcal{R}} = 0$.
Also by the multiplicity condition of $\overline{f}$, the children of $\cR$ are $r$ pairwise twin $\ft_1, \cdots, \ft_r$ and $(2g+1-2r)$ singletons.
Hence the cluster picture is
\begin{align*}
    \lcrc{\bullet, \cdots, \bullet, \lcrc{\bullet, \bullet}_{d_1}, \cdots \lcrc{\bullet, \bullet}_{d_r}  }_0
\end{align*}
for the depths $d_i > 0$.
So, the unique principal cluster is $\cR$ and the proper clusters are $\cR, \ft_1, \cdots, \ft_r$.

Since inertia acts trivially on the residue field, $\ft_i$ and $\cR$ are $I_F$-invariant.
For the unique principal cluster $\cR$, we know $d_{\cR} = 0$. Also, 
\begin{align*}
    \nu_{\cR} = |\cR| d_{\cR}  = 0 \in 2\bZ.
\end{align*}
Thus, by Theorem \ref{thm: BBB S 5} (1), it remains only to show that $e(F(\mathcal R)/F)\leq 2.$
Consider the permutation action
\begin{align*}
\rho_{\mathcal R}\colon I_F\longrightarrow \operatorname{Aut}(\mathcal R)
\end{align*}
induced by the action of inertia on the roots of $f$. Since $p>2g+1=|\mathcal R|$, the wild inertia subgroup, which is a pro-$p$ group, acts trivially on $\mathcal R$. Hence $\rho_{\mathcal R}$ factors through tame inertia.

Each singleton child of $\mathcal R$ is fixed pointwise by $I_F$, while on each twin $\mathfrak t_i$, inertia can only act trivially or interchange its two roots. Therefore
\begin{align*}
\rho_{\mathcal R}(I_F)
\subseteq
\prod_{i=1}^r \operatorname{Aut}(\mathfrak t_i)
\cong (\mathbb Z/2\mathbb Z)^r.
\end{align*}
On the other hand, every finite quotient of tame inertia is cyclic. Thus $\rho_{\mathcal R}(I_F)$ is a cyclic subgroup of $(\mathbb Z/2\mathbb Z)^r$, and consequently $e(F(\mathcal R)/F) = |\rho_{\mathcal R}(I_F)|\leq 2.$
Therefore $C$, equivalently $J$, has semistable reduction.

Since $J$ is semistable, Theorem \ref{thm: BBB S 5} (2) gives the toric rank.
The twins, which are the even clusters in $\cR$, are not \"ubereven.
Since $\cR$ is non-\"ubereven, the toric rank is $r$, and the abelian rank $a = g-r$.
Since the partition of $f$ is $(2, \ldots, 2, 1, \ldots, 1)$, we have $a_0 = g-r$ and $t_0 = r$.

For (3), since every multiple root of $\overline f$ has multiplicity two, the special fiber $\overline C$ is reduced and has only ordinary double points. Hence $\mathcal C/\mathcal O_F$ is a semistable model of $C$.
By \cite[Proposition 2.8]{BW}, for every prime $\ell$ different from the residue characteristic there is a Frobenius-equivariant isomorphism
\begin{align*}
H^1_{\mathrm{\acute et}}
    (C_{\overline F},\mathbb Q_\ell)^{I_F}
\cong
H^1_{\mathrm{\acute et}}
    (\overline C_{\overline k},\mathbb Q_\ell).
\end{align*}
Equivalently, by \cite[Corollary 2.5]{BW}, the local Euler factor of $C/F$ is the inverse of the numerator of the zeta function of $\overline C$.
Therefore, by the Grothendieck--Lefschetz trace formula, for every $m\geq1$,
\begin{align*}
a_{\mathfrak p^m}(C)
&=
\operatorname{Tr}\left(
\operatorname{Frob}_{q_{\mathfrak p}}^m
\mid
H^1_{\mathrm{\acute et}}
(\overline C_{\overline k},\mathbb Q_\ell)
\right) = 
q_{\mathfrak p}^m+1
-\#\overline C(\mathbb F_{q_{\mathfrak p}^m}),
\end{align*}
as claimed.
\end{proof}

For even or cotwin $\fs$, we define
\begin{align*}
    \epsilon_{\fs} : G_F \longrightarrow \lcrc{\pm 1}, \qquad 
    \epsilon_{\fs}(\sigma) \equiv \frac{\sigma(\theta_{\fs^*})}{\theta_{\sigma(\fs)^*}}
\end{align*}
and $0$ otherwise, where $\theta_{\fs}$ is a choice of
\begin{align*}
    \sqrt{c \prod_{r \not\in \fs} (z_{\fs} - r) }.
\end{align*}

Let $C/F$ be a hyperelliptic curve defined by $y^2 = f(x)$.
In \cite[\S 11]{BBB}, the decomposition
\begin{align*}
    H^1_{\mathrm{\acute et}}(C/\overline{F},\bQ_\ell)
    = H^1_{\mathrm{ab}} \oplus \lbrb{H^1_{\mathrm{t}} \otimes \mathrm{sp}(2)}
\end{align*}
is introduced when $\ell$ is not equal to the characteristic of the residue field of $F$.
The N\'eron ranks satisfy
\[
    2a(f)=\dim(H^1_{\mathrm{ab}})^I,
    \qquad
    t(f)=\dim(H^1_{\mathrm t})^I.
\]
Using the cluster picture, we can describe both abelian and toric parts.
Following \cite[p.849, Notation 11.4]{BBB}, we define
\begin{align*}
    G_{\fs} := \mathrm{Stab}_{G_F}(\fs), \qquad I_{\fs} := \mathrm{Stab}_{I_F}(\fs).
\end{align*}
We denote $\bar{\epsilon}_{\fs}$ by the restriction of $\epsilon_{\fs}$ to $G_{\fs}$.
For $C/F$, we also define
\begin{align*}
    X := \lcrc{\fs : \textrm{proper, non-\"ubereven}}, \qquad 
    Y := \lcrc{\fs : \textrm{principal, non-\"ubereven}}.
\end{align*}
We note that $Y \subset X$.

\begin{theorem} \label{thm: BBB 11.1}
Let $C/F$ be a hyperelliptic curve and $\ell \neq p$.
\begin{enumerate}
    \item For each $\fs \in Y$, there exists a continuous $\ell$-adic representation $V_{\fs}$ such that
    \begin{align*}
        H^1_{\mathrm{ab}} = \bigoplus_{\fs \in Y / G_K} \mathrm{Ind}_{G_{\fs}}^{G_K} V_{\fs}, \qquad 
        H^1_{\mathrm{t}} = 
        \lbrb{\bigoplus_{\fs \in X /G_K} \mathrm{Ind}_{G_{\fs}}^{G_K} \bar{\epsilon}_{\fs} } \ominus \bar{\epsilon}_{\cR}.
    \end{align*}
    \item Let $\gamma_{\fs} : I_{\fs} \to \overline{\bQ_{\ell}}^\times$ be any character whose order is prime-to-$p$ part of the denominator of $[I_K : I_{\fs}]\widetilde{\lambda}_{\fs}$.
    Then, for $\fs \in Y$, we have
    \begin{align*}
        V_{\fs} \cong \gamma_{\fs} \otimes (\bQ_{\ell}[\widetilde{\fs}] \ominus 1) \ominus\bar{\epsilon}_{\fs}
    \end{align*}
    as $I_{\fs}$-representation.
\end{enumerate}
\end{theorem}
\begin{proof}
    This is \cite[Theorem 11.1]{BBB}.
\end{proof}

\begin{proposition} 
\label{lem:rank-comparison}
Let $f\in R[x]$ be monic and separable of degree $n$,
and let $\lambda$ be the geometric root multiplicity partition
of $\overline f$.
Then $a(f) \geq a_0(\lambda)$ and $t(f) \geq t_0(\lambda)$.
\end{proposition}

\begin{proof}

If $\overline f$ has only one geometric root, then
$\lambda=(n)$ and $a_0(\lambda)=t_0(\lambda)=0$, so the inequalities are immediate.

Otherwise, since the roots of $f$ are integral, $d_{\mathcal R}=0$ and the children of the cluster $\mathcal R$ are precisely the blocks of roots with the same reduction.
Each such block is fixed setwise by $I_F$.
Let $c_{\odd}(\lambda)$ be the number of odd children of $\mathcal R$.
Then, 
\begin{align*}
    a_0(\lambda) = g - \sum_{j=1}^s \lfrf{\frac{m_j}{2}}
    = g -  \frac{n - c_{\odd}}{2}.
\end{align*}
Hence, we have $c_{\odd}(\lambda) = 2a_0(\lambda)+1$.

If $a_0(\lambda)>0$, then $c_{\odd}(\lambda)\geq3$.
Hence $\mathcal R$ cannot have a child of size $2g$, and so it is not a cotwin. 
Since $|\mathcal R|=2g+1>2$, the cluster $\mathcal R$ is proper and is not a twin. 
Thus $\mathcal R$ is principal.
Since $f$ is monic and $d_{\mathcal R}=0$, one has $\widetilde{\lambda}_{\cR} = 0$.

Assume $a_0(\lambda)>0$. 
As observed above, the top cluster $\mathcal R$ is principal and not \"ubereven, so $\mathcal R\in Y$.
Moreover, $G_{\cR} = G_F$ and $I_{\cR} = I_F$.
Thus the summand of $H^1_{\mathrm{ab}}$ in Theorem \ref{thm: BBB 11.1} (1) corresponding to $\mathcal R$ is simply $V_{\mathcal R}$.
We now determine $V_{\mathcal R}|_{I_F}$ using Theorem \ref{thm: BBB 11.1} (2). 
Since $\widetilde\lambda_{\mathcal R}=0$, the character $\gamma_{\mathcal R}$ is trivial.
Recall that $\widetilde{\mathcal R}$ is the set of odd children of $\mathcal R$.  
These children are the blocks of roots having the same reduction, and hence each of them is fixed setwise by $I_F$.
Therefore
\[
    \mathbb Q_\ell[\widetilde{\mathcal R}]
    \cong
    1^{\oplus c_{\odd}(\lambda)}
\]
as $I_F$-representations.  
Finally, since $\mathcal R$ is odd and is not a cotwin, $\epsilon_{\mathcal R}$ vanishes.
Consequently,
\[
    V_{\mathcal R}|_{I_F}
    \cong
    \mathbb Q_\ell[\widetilde{\mathcal R}]\ominus 1
    \cong
    1^{\oplus(c_{\odd}(\lambda)-1)}.
\]
Since $V_{\mathcal R}$ is a direct summand of $H^1_{\mathrm{ab}}$,
we obtain
\[
    2a(f) = \dim(H^1_{\mathrm{ab}})^I
    \ge c_{\odd}(\lambda)-1
    =2a_0(\lambda).
\]

Consider the toric part.
We recall that $\bar{\epsilon}_{\fs}$ is zero when $\fs$ is odd and not cotwin, and the possible odd cotwin cluster is $\cR$ itself.
Let $X' \subset X$ be the set of even and non-\"ubereven clusters in $\cR$. 
Then for each $\fs \in X \setminus X'$, the character $\epsilon_{\fs}$ vanishes with the possible exception $\cR$.
Hence by Theorem \ref{thm: BBB 11.1} (1),
\begin{align*}
    H_{\mathrm{t}}^1 = \lbrb{\bigoplus_{\fs \in X /G_F} \mathrm{Ind}_{G_{\fs}}^{G_F} \bar{\epsilon}_{\fs} } \ominus \bar{\epsilon}_{\cR}
    = \bigoplus_{\fs \in X' /G_F} \mathrm{Ind}_{G_{\fs}}^{G_F} \bar{\epsilon}_{\fs}
\end{align*}
as $G_{F}$-representation.
By restricting it to $I_F$, we have
\begin{align} \label{eqn: toric X'}
    H_{\mathrm{t}}^1  \cong \bigoplus_{\fs \in X'/I_F } \mathrm{Ind}_{I_{\fs}}^{I_F} \lbrb{ \bar{\epsilon}_{\fs}|_{I_{\fs}} }.
\end{align}

To show that $t(f) \geq t_0(\lambda)$, we will obtain at least one contribution to the toric rank from each even child of $\cR$.

Let $\fs$ be an even child of $\cR$.
First, note that $\fs$ is not cotwin, since it has no child of size $2g$.
If $\fs$ is not \"ubereven, then $\fs \in X'$.
If $\fs$ is \"ubereven, then $\fs \not\in X'$ but all of its children are even.
We then proceed recursively: along each branch, we stop if an even child is non-\"ubereven.
In this way, we obtain a finite tree of even clusters rooted at $\fs$.
Let $\mathcal{D}(\fs)$ be the set of the first non-\"ubereven clusters in each branch.

We note that $\mathcal{D}(\fs)$ is nonempty and for each $\ft \in \mathcal{D}(\fs)$, $\ft^* = \fs$ by definition.
Indeed, first suppose that $\fs$ itself is not \"ubereven.
Then $\ft=\fs$, and the parent of $\fs$ is $\cR$.
Since $\fs$ is not cotwin, $|\cR|=2g+1$ is odd, $\cR$ is not \"ubereven, and hence $\fs^*=\fs$.
Now suppose that $\fs$ is \"ubereven.
For $\ft\in\mathcal D(\fs)$, every cluster strictly between $\ft$ and $\fs$ is \"ubereven by construction.
On the other hand, the parent of $\fs$ is $\cR$, which is not \"ubereven since it is odd.
Thus $\fs$ is the smallest cluster containing $\ft$ whose parent is not \"ubereven, and therefore $\ft^*=\fs$ by the definition of $\ft^*$.

Since the children of $\cR$ are precisely the blocks of roots having the same reduction, $\fs$ is fixed setwise by $I_F$.
Since $p>n$, the integer $|\fs|$ is a unit in $R$,
\begin{align*}
    \frac{1}{|\fs|}
    \sum_{r\in\fs} r \in \cO_{\overline{F}}.
\end{align*}
This is a centre $z_{\fs}$ of $\fs$. Indeed, for any $r_0 \in \fs$, 
\begin{align*}
    v\lbrb{ \frac{1}{|\fs|} \sum_{r \in \fs} (r - r_0) }
    = v\lbrb{\sum_{r \in \fs} (r - r_0) } \geq d_{\fs}.
\end{align*}

Since $f$ is monic, we have
\begin{align*}
    \theta_{\fs}^2
    =
    \prod_{r\notin\fs}
    (z_{\fs}-r).
\end{align*}
Every root outside $\fs$ has reduction different from that of the roots in $\fs$.
Hence $v(z_{\fs} - r) = 0$ and $z_{\fs} - r \in \cO_{\overline{F}}^\times$ for each $r \not\in \fs$.
On the other hand, $z_{\fs}$ is $I_F$-fixed by definition, and $\cR \setminus \fs$ is $I_{F}$-stable.
Therefore, $\theta_\fs^2$ is an element in $\cO_{\overline{F}}^\times$ which is fixed by $I_F$.
Since $p \neq 2$ and the residue field of the maximal unramified extension $F^{\mathrm{nr}}$ is algebraically closed, we may choose a square root $\theta_{\fs}$ in $\cO_{F^{\mathrm{nr}}}^\times$.
Therefore, $\theta_{\fs}$ is also $I_F$-fixed.

Let $\ft \in \mathcal{D}(\fs)$.
We recall that $\ft^* = \fs$.
Since $\theta_{\fs}$ is fixed by $I_F$, we have $\epsilon_{\ft}|_{I_{\ft}} = 1$.
It follows from Frobenius reciprocity that
\begin{align*}
    \dim
    \left(
        \operatorname{Ind}_{I_{\ft}}^{I_F}
         1
    \right)^{I_F}
    =
    \dim 1^{I_{\ft}}
    =
    1.
\end{align*}
Thus every $I_F$-orbit in $\mathcal D(\fs)$ contributes one dimension to $\left(H_{\mathrm t}^1\right)^{I_F}.$
In particular, since $\mathcal D(\fs)$ is nonempty, each even child $\fs$ of $\cR$ contributes at least one dimension.

Finally, if $\fs_1$ and $\fs_2$ are distinct even children of $\cR$, then the clusters in $\mathcal D(\fs_1)$ and $\mathcal D(\fs_2)$ lie in disjoint branches of the subtree of even clusters rooted at the children of $\cR$.
Since $\fs_i$ is $I_F$-stable, these branches cannot be permuted by $I_F$.
Hence, the corresponding clusters lin in distinct $I_F$-orbits and give  distinct summands in \eqref{eqn: toric X'}.
Therefore,
\begin{align*}
    t(f)
    =
    \dim
    \left(H_{\mathrm t}^1\right)^{I_F}
    \geq
    \#\{
        \text{even children of }\cR
    \}
    =
    t_0(\lambda).
\end{align*}
This proves the result.
\end{proof}

\subsection{Reduction} \label{subsec: reduction}

Let $K$ be a number field, $\fp$ a prime of $K$, $\cO_{K, \fp}$ the ring of integers of $K_{\fp}$, and $\kappa(\fp)$ the residue field of $\cO_{K, \fp}$.
Given a prime $\mathfrak{p}$ of $K$ lying over a rational prime $p>2g+1$, each $\bfa\in\mathcal{P}(w_g)(K_\mathfrak{p})$ admits a representative $(a_2,a_3,\ldots,a_{2g+1})\in\mathbb{A}^{2g}(\mathcal{O}_{K,\mathfrak{p}})$ such that $\ord_\mathfrak{p}(a_i)<2i$ for at least one $i$.
Such a representative is unique up to the weighted action of $\cO_{K, \fp}^\times$.
Put
\begin{align*}
    U_{g, \fp} := \cO_{K, \fp}^{2g} \setminus \prod_{i=2}^{2g+1} \pi_{\fp}^{2i} \cO_{K, \fp}.
\end{align*}
Then, the above representatives of $\bfa$ are in $U_{g, \fp}$.
Using this representative, we define a function
\begin{align} \label{eqn: def mod fp}
    \xymatrixcolsep{1.75pc}\xymatrix{\psi_\mathfrak{p}:\mathcal{P}(w_g)(K_\mathfrak{p}) \ar[r] & \mathbb{A}^{2g}(\kappa(\mathfrak{p}))/\kappa(\mathfrak{p})^\times \ar[r]^-\sim & \mathcal{P}(w_g)(\kappa(\mathfrak{p}))\coprod\{[0]\} & \bfa \ar@{|->}[r] & [a_i\bmod\mathfrak{p}]_{i=2,\cdots,2g+1}\rlap{\ .}}
\end{align}
For a weighted primitive representative corresponding to $f(x)=x^{2g+1}+\sum_{i=2}^{2g+1}a_i x^{2g+1-i}$, the reduction $\psi_{\mathfrak p}(\mathbf a)$ records precisely the coefficients of $\overline f$. 
Thus, it determines the naive special fiber of $\mathcal C_f$, given by $y^2=\overline f(x)$, up to isomorphism over $\kappa(\mathfrak p)$.
In what follows, we write
\begin{align*}
    \mathcal{P}(w_g)(\kappa(\mathfrak{p}))^\ast:=\mathcal{P}(w_g)(\kappa(\mathfrak{p}))\coprod\{[0]\}.
\end{align*}
Note that the representative in $\mathbb{A}^{2g}(\mathcal{O}_{K,\mathfrak{p}})$ gives the minimal Weierstrass model
$y^2 = f(x)$ over $K_\mathfrak{p}$ where $f(x) = x^{2g+1} + a_{2} x^{2g-1} + \cdots + a_{2g+1}.$

We will use the height defined on $\cP(w)(K)$, following \cite{BN, CJP, Phi2}.
Let $M_{K, 0}$ and $M_{K, \infty}$ be the set of finite, infinite places of $K$.
For $v \in M_{K, 0}$, we define the absolute value by $|\pi_v|_v = \frac{1}{N(\fp_v)}$ where $\pi_v$ is a uniformizer, and by $|a|_v := |i_v(a)|^{[K_v : \bR]}$ for $v \in M_{K, \infty}$, where $i_v : K \to K_v$.
For $(x_0, \cdots, x_n) \in K^{n+1} \setminus \lcrc{0}$ we define
\begin{align*}
    |(x_0, \cdots, x_n)|_{w, v} :=
    \left\{ \begin{array}{lll}
        \max\limits_{0 \leq i \leq n} \lcrc{|\pi_v|_v^{ \lfloor \frac{\ord_v(x_i)}{w_i} \rfloor }} & \textrm{for } v \in M_{K, 0},  \\
        \max\limits_{0 \leq i \leq n} \lcrc{|x_i|_v^{\frac{1}{w_i}}} &  \textrm{for } v \in M_{K,\infty}.
    \end{array}
    \right.
\end{align*}
Then, for $[x_0, \cdots, x_n] \in \cP(w)(K)$, we define
\begin{align*}
    H_{w, K}([x_0, \cdots, x_n]) :=
    \prod_{v \in M_K} |(x_0, \cdots, x_n)|_{w, v}.
\end{align*}
We note that $H_{w,K}$ is well defined on $\cP(w)(K)$.
Finally, we define
\begin{align*}
    \cP(w)(K)(X) := \lcrc{\bfa \in \cP(w)(K) : H_{w, K}(\bfa) \leq X}.
\end{align*}

For a field $k$ and $\bfa = (a_2, a_3, \cdots, a_{2g}, a_{2g+1}) \in \bA^{2g}(k)$, we define
\begin{align*}
    f_{\bfa}(x) := x^{2g+1} + a_2x^{2g-1} + \cdots + a_{2g}x + a_{2g+1}.
\end{align*}

\begin{lemma} \label{lem: disc well def}
The discriminant $\Delta(f_\mathbf{a})$ of the polynomial $f_{\bfa}$ is a weighted homogeneous polynomial of weighted degree $4g(2g+1)$.
In particular, the set 
\begin{align*}
    \lcrc{\bfa \in \cP(w_g)(k) : \Delta(f_{\bfa}) \neq 0 }
\end{align*}
is well-defined.
\end{lemma}
\begin{proof}
We have $f_{\lambda \cdot \bfa}(x) = \lambda^{2(2g+1)}f_{\bfa}(\lambda^{-2}x)$.
Since the roots of $f_{\lambda\cdot \bfa}$ are $\lambda^2 \alpha_1,\dots,\lambda^2 \alpha_{2g+1}$, we have
\begin{align*}
    \Delta(f_{\lambda\cdot \bfa})
    &= \prod_{i<j} (\lambda^2 \alpha_i-\lambda^2 \alpha_j)^2 
    = \lambda^{4\binom{2g+1}{2}} \prod_{i<j} (\alpha_i-\alpha_j)^2 = \lambda^{4g(2g+1)}\Delta(f_\bfa).
\end{align*}
Thus, $\Delta$ is a weighted homogeneous polynomial of weighted degree $4g(2g+1)$, and the condition $\Delta(f_\bfa)\neq 0$ is well-defined on $\cP(w_g)(k)$.
\end{proof}
For $\bfa\in\cP(w_g)(K)$ and a prime $\fp$ of $K$, we write $\Delta(\bfa)\not\equiv 0\pmod{\fp}$ if $\psi_{\fp}(\bfa)=\bfz\in\cP(w_g)(\kappa(\fp))$ and $\Delta(f_{\bfz})\neq 0$.
By Lemma \ref{lem: disc well def}, this condition is well-defined.

\smallskip 

Let $\cH_{g, 1}$ denote the moduli stack of hyperelliptic curves of genus $g$ with a marked Weierstrass point. 
By \cite[Proposition 5.9]{HP23}, there is a compactification of $\cH_{g, 1}$ which is isomorphic to $\cP(w_g)$.
Under the isomorphism, $\cH_{g, 1}$ is isomorphic to the open substack parametrizing the smooth hyperelliptic curves.
Hence, we have
\begin{align*}
    \cH_{g, 1} \cong \cP(w_g) \setminus V(\Delta)
\end{align*}
over any field of characteristic $0$ or greater than $2g+1$.
Using this isomorphism, we define a height $H_{w, K}$ on $\cH_{g, 1}(K)$ for a number field $K$, and let
\begin{align*}
    \cH_{g, 1}(K)(X) := \lcrc{C \in \cH_{g,1}(K) : H_{w, K}(C) \leq X}.
\end{align*}

Let $q$ be a prime power satisfying $\kappa(\fp) \cong \bF_q$.
For a tuple of positive integers $w$, we define $|w|$ to be the sum of its entries.
For $\bfz = [z_2, z_3, \cdots, z_{2g+1}]$, an element of $\cP(w_g)(\bF_q)$, we define 
\begin{align*}
   I(\bfz) := \lcrc{i \in \lcrc{2, 3, \cdots, 2g+1} : z_i \neq 0}, \qquad 
   d(\bfz) := \gcd\lcrc{w_i : i \in I(\bfz)}.
\end{align*}
Then, 
\begin{align*}
    \Stab_{w_g}(\bfz) 
    = \lcrc{\lambda \in \bF_q^\times : \lambda^{w_i}z_i = z_i \textrm{ for }i = 2, \cdots, 2g+1}
\end{align*}
and by the orbit-stabilizer theorem, the number of affine representatives of $\bfz$ is 
\begin{align}  \label{eqn: num affine rep}
    \frac{q-1}{\# \Stab_{w_g}(\bfz)} = \frac{q-1}{\# \mu_{d(\bfz)}(\bF_q)}.
\end{align}

For a function $f : \bA^{2g}(\bF_q) \setminus \lcrc{0} \to \bC^\times$, which is invariant under $w_g$-weighted $\bF_q^\times$-action, we have
\begin{align} \label{eqn: cP Fq to A Fq}
    \sum_{\bfz \in \cP(w_g)(\bF_q)} \frac{q-1}{\# \mu_{d(\bfz)}(\bF_q)} f(\bfz)
    = \sum_{\bfz \in \bA^{2g}(\bF_q) \setminus \lcrc{0}} f(\bfz)
\end{align}
by (\ref{eqn: num affine rep}).

\begin{proposition} \label{prop: counting P loc cond}
Let $\psi_{\fp} : \cP(w_g)(K_{\fp}) \to \cP(w_g)^*(\bF_q)$ be the mod $\fp$ reduction defined by (\ref{eqn: def mod fp}) and let $D$ be the degree $[K:\bQ]$.
For $\bfz \in \cP(w_g)(\bF_q)$, there is a constant $\kappa$ depending only on $w_g$ and $K$ satisfying
\begin{align*}
    \# \lcrc{ \bfa \in \cP(w_g)(K)(X) : \psi_{\fp}(\bfa) = \bfz}
    = \frac{q-1}{\# \mu_{d(\bfz)}(\bF_q)} \frac{1}{q^{2g}} \frac{1}{1 - q^{-|w_g|}} \kappa X^{|w_g|}  +  O\lbrb{q^{2-2g}X^{|w_g| - \frac{4}{D} } \log X}.
\end{align*}
Also, 
\begin{align*}
    \# \lcrc{ \bfa \in \cP(w_g)(K)(X) : \psi_{\fp}(\bfa) = [0]}
    = \frac{q^{-2g} - q^{-|w_g|}}{1 - q^{-|w_g|}} \kappa X^{|w_g|}
    + O(q X^{|w_g| - \frac{4}{D}} \log X).
\end{align*}
\end{proposition}
\begin{proof}
Let 
\begin{align*}
    \Omega_{\fp, \bfz} := \lcrc{\bfa \in \cP(w_g)(K_{\fp}) : \psi_{\fp}(\bfa) = \bfz}, \qquad 
    \Omega_{\fp, \bfz}^{\aff} := \lcrc{\bfa \in K_{\fp}^{2g} \setminus \lcrc{0} : [\bfa] \in \Omega_{\fp, \bfz} }.
\end{align*}
and let $A_{\fp, \bfz} := \Omega_{\fp, \bfz}^{\aff} \cap U_{g, \fp}.$
For $\bfz = [z_2, \ldots, z_{2g+1}] \in \cP(w_g)(\bF_q)$, choose lifts $\widetilde{z}_i \in \cO_{K, \fp}$ of the chosen homogeneous coordinates $z_i$ and let $\widetilde{\bfz} = (\widetilde{z}_2, \ldots, \widetilde{z}_{2g+1})$.
Then, $\psi_{\fp}([\widetilde{\bfz}]) = \bfz$.
Let
\begin{align*}
    B_{\fp, \bfz} := \prod_{i=2}^{2g+1} (\widetilde{z}_i + \fp \cO_{K, \fp}).
\end{align*}
By the definition of $\psi_{\fp}$, for $\bfa = (a_2, \ldots, a_{2g+1}) \in U_{g, \fp}$, $\psi_{\fp}([a_2, \ldots, a_{2g+1}]) = \bfz$ if and only if
\begin{align*}
    [\overline{a}_2, \ldots, \overline{a}_{2g+1}] = [z_2, \ldots, z_{2g+1}]
\end{align*}
where $\overline{a}_i \in \cO_{K, \fp}/\fp \cong \bF_q$ is $a_i$ modulo $\fp$.
This is again equivalent to $(\overline{a}_2, \ldots, \overline{a}_{2g+1}) = \overline{u}*_{w_g}(z_2, \ldots, z_{2g+1})$ for some $\overline{u} \in \bF_q^\times$.
On the other hand, $\overline{u}_1*_{w_g} (z_2, \ldots, z_{2g+1}) = \overline{u}_2*_{w_g}(z_2, \ldots, z_{2g+1})$ if and only if $\overline{u}_1/\overline{u}_2 \in \Stab_{w_g}(\bfz) = \mu_{d(\bfz)}(\bF_q)$.
Hence, 
\begin{align*}
    A_{\fp, \bfz} =   \bigsqcup_{\substack{\bar{u} \in \bF_q^\times/\mu_{d(\bfz)}(\bF_q)}} u*_{w_g} B_{\fp, \bfz}
\end{align*}
where $u \in \cO_{K, \fp}^\times$ denotes any lift of $\overline{u}$.
We note that $B_{\fp, \bfz}$ itself depends on the choice of the affine representative $(z_2, \ldots, z_{2g+1})$, whereas the above disjoint union, and hence $A_{\fp, \bfz}$, is independent of this choice.
Therefore, 
\begin{align*}
    \Omega_{\fp, \bfz}^{\aff} \cap \cO_{K, \fp}^{2g} 
    = \bigsqcup_{k \geq 0}\pi_{\fp}^k*_{w_g}A_{\fp, \bfz}
    = \bigsqcup_{\mu \in \bF_{q}^\times/\mu_{d(\bfz)} (\bF_q)} \bigsqcup_{k \geq 0} u \pi_{\fp}^k *_{w_g} B_{\fp, \bfz}
\end{align*}
where $B_{\fp, \bfz}$ is an affine box of the form required for the local condition in \cite[\S 6.1]{CJP}.
Therefore, $\Omega_{\fp}$ satisfies the condition of \cite[Theorem 6.18]{CJP}.
By \cite[Theorem 6.18]{CJP} for $f = \mathrm{id}$ and $\Omega_{\fp}$, we have
\begin{align*}
    \# \lcrc{\bfa \in \cP(w_g)(K) : H_{w_g, K}(\bfa) \leq X, \bfa \in \Omega_{\fp} }
    = \kappa m_{\fp}(\Omega_{\fp}^{\aff} \cap \cO_{K, \fp}^{2g}) X^{|w_g|}
    + O\lbrb{ q^{2-2g}X^{|w_g| - \frac{4}{D} } \log X}.
\end{align*}
Then,
\begin{align*}
    m_{\fp} \lbrb{\Omega_{\fp}^{\aff} \cap \cO_{K, \fp}^{2g}}
    &= \sum_{\mu \in \bF_{q}^\times/\mu_{d(\bfz)}(\bF_q)} \sum_{k \geq 0} 
    m_{\fp} \lbrb{ \pi_{\fp}^k *_{w_g} B_{\fp, \bfz}   } 
    = \frac{q-1}{\# \mu_{d(\bfz)}(\bF_q)} \sum_{k \geq 0} q^{-k|w_g|} m_{\fp} \lbrb{ B_{\fp, \bfz}   } \\
    &= \frac{q-1}{\# \mu_{d(\bfz)}(\bF_q)}  \sum_{k \geq 0} q^{-k|w_g|} q^{-2g}\\
    &= \frac{q-1}{\# \mu_{d(\bfz)}(\bF_q)} \frac{1}{q^{2g}} \frac{1}{1 - q^{-|w_g|}}.
\end{align*}
This proves the case of $\psi_{\fp}(\bfa) = \bfz$.

We now consider the exceptional fiber $[0]$.
By the weighted projective height counting result used above, we have
\begin{align}
\label{eqn: Pwg height points}
    \#\cP(w_g)(K)(X)
    =\kappa X^{|w_g|}
    +O_{K,g}\left(
        X^{|w_g|-\frac4D}\log X
      \right).
\end{align}
By \eqref{eqn: cP Fq to A Fq}, we have
\begin{align*}
\sum_{\bfz\in\cP(w_g)(\bF_q)}
\frac{q-1}{\#\mu_{d(\bfz)}(\bF_q)}
=q^{2g}-1.
\end{align*}
Since $\#\cP(w_g)(\bF_q)\ll_g q^{2g-1}$, summing the preceding formula over all $\bfz\in\cP(w_g)(\bF_q)$ gives
\begin{align*}
&\#\lcrc{
\bfa\in\cP(w_g)(K)(X):
\psi_{\fp}(\bfa)\neq[0]}=
\frac{1-q^{-2g}}{1-q^{-|w_g|}}
\kappa X^{|w_g|}
+O_{K,g}\lbrb{
qX^{|w_g|-\frac4D}\log X}.
\end{align*}
Subtracting this from the unrestricted count \eqref{eqn: Pwg height points}, we obtain
\begin{align*}
&\#\lcrc{
\bfa\in\cP(w_g)(K)(X):
\psi_{\fp}(\bfa)=[0]} =
\frac{q^{-2g}-q^{-|w_g|}}
{1-q^{-|w_g|}}
\kappa X^{|w_g|}
+O_{K,g}\lbrb{
qX^{|w_g|-\frac4D}\log X}.
\end{align*}
This completes the proof.
\end{proof}

\begin{lemma} \label{lem: Hg1 counting}
Let $K$ be a number field of degree $D$ and let $X > 0$.
There is a positive constant $\kappa$ depending only on $g$ and $K$ such that 
\begin{align*}
    \# \cH_{g, 1}(K)(X) = \kappa X^{|w_g|} + O_{K, g}\lbrb{X^{|w_g| - \frac{4}{D}} \log X}.
\end{align*}
\end{lemma}
\begin{proof}
By \eqref{eqn: Pwg height points}, it remains to estimate the number of points with zero discriminant.
Put $n=2g+1$ and $w_g'=(4,6,\ldots,4g)$.
We have $\Delta(0, \ldots, 0, a_n) = \pm n^n a_n^{n-1} \neq 0$ when $a_n \neq 0$.
Thus every point of $V(\Delta)$ has at least one nonzero coordinate among $a_2,\ldots,a_{n-1}$, and the projection
\begin{align*}
    \pi : V(\Delta) \longrightarrow \cP(w_g'), \qquad 
    [a_2, \ldots, a_n] \mapsto [a_2, \ldots, a_{n-1}]
\end{align*}
is well-defined over $K$.
For any point of $\cP(w_g')(K)$, fixing a representative of its coordinates leaves at most $n-1$ possible values of $a_n$ satisfying $\Delta(f_{\bfa})=0$.
Hence every fiber of $\pi$ on $K$-points has cardinality at most $n-1$.

Moreover, the definition of the height gives
\begin{align*}
    H_{w_g',K}(\pi(\bfa)) \le H_{w_g,K}(\bfa),
\end{align*}
since omitting a coordinate cannot increase any of the local height factors.
Consequently,
\begin{align*}
    &\#\{\bfa\in\cP(w_g)(K)(X):
              \Delta(f_{\bfa})=0\} 
              \leq  (n-1)\#\cP(w_g')(K)(X) 
              \ll X^{|w_g'|} =X^{|w_g|-(4g+2)}.
\end{align*}
Since $4g+2>4/D$, subtracting the discriminant-zero locus from \eqref{eqn: Pwg height points} yields the desired asymptotic formula.
\end{proof}

\section{Countings and local probabilities} \label{sec: proof main prob}

\subsection{Finite field counting}

\label{subsec: FF counting}



\begin{lemma}
Let $w=(w_0,\dots,w_n)$ be a tuple of positive integers.
Then
\begin{align*}
    \#\cP(w)(\bF_q)
    =
    \frac{1}{q-1}
    \sum_{\lambda\in \bF_q^\times}
    \bigl(q^{m(\lambda)}-1\bigr)
    = \frac{1}{q-1} \sum_{d \mid (q-1)} \phi(d)\lbrb{q^{m_d} - 1},
\end{align*}
where
\[
m(\lambda):=\#\{\,0\le i\le n : \lambda^{w_i}=1\,\}, \qquad 
m_d := \# \lcrc{0 \leq i \leq n  : d \mid w_i}.
\]
\end{lemma}

\begin{proof}
We recall Burnside's lemma, which says that for a finite group $G$ and a finite $G$-set $S$,
\begin{align*}
    \#(S/G) = \frac{1}{\# G} \sum_{g\in G} \# S^g, \qquad 
    S^g := \lcrc{s \in S : g.s = s}.
\end{align*}
Let $X:=\bA^{n+1}\setminus\{0\}.$
We will apply Burnside's Lemma to the $\bG_m(\bF_q)$-set $X(\bF_q)$.
We note that $ \bfz = (z_0, \cdots, z_n) \in X(\bF_q)$ is fixed by $\lambda$ if and only if $\lambda^{w_i}z_i = z_i$ for all $i$.
So for $\bfz \in X(\bF_q)^{\lambda}$, we have either $z_i=0$ or $\lambda^{w_i}=1$. 
Hence the coordinates $z_i$ are free exactly for those $i$ with $\lambda^{w_i}=1$. 
Therefore, 
\[
\#X(\bF_q)^\lambda=q^{m(\lambda)}-1, \qquad 
\#\cP(w)(\bF_q)
=
\frac{1}{q-1}
\sum_{\lambda\in \bF_q^\times}
\bigl(q^{m(\lambda)}-1\bigr).
\]
We note that if $\lambda$ has order $d$, then $\lambda^{w_i} = 1$ if and only if $d \mid w_i$.
Also, there are $\phi(d)$ elements of order $d$ in $\bF_q^\times$.
So we have the second expression.
\end{proof}

Using the second one, we have
\begin{align} \label{eqn: cP wg Fq counting}
    \#\cP(w_g)(\bF_q) = \frac{2(q^{2g}-1)}{(q-1)} + O\lbrb{q^{\lfloor\frac{2g+1}{2} \rfloor- 1} }.
\end{align}

\begin{lemma} \label{lem: num sqf poly}
Let $a_n$ be the number of monic squarefree polynomials of degree $n$ in $\bF_q[x]$.
Then, $a_0 = 1, a_1 = q$, and $a_n = q^n - q^{n-1}$.
\end{lemma}
\begin{proof}
Since every monic polynomial can be written as $s \cdot h^2$ for a monic squarefree $s$ and a monic $h$, we have
\begin{align*}
    \sum_{n \geq 0} q^n t^n = \lbrb{\sum_{m \geq 0}a_mt^m} \lbrb{\sum_{r \geq 0} q^r t^{2r}}
\end{align*}
by considering the generating series. Hence, we have
\begin{align*}
    \sum_{m \geq 0} a_m t^m 
    &= (1 - qt^2)(1 + qt + q^2t^2 + \cdots)= 1+ qt + (q^2 - q)t^2 + (q^3 - q^2)t^3 + \cdots.
\end{align*}
Considering the exceptional case $n = 1$, we obtain the result. 
\end{proof}

\begin{corollary} \label{cor: sqf poly variants}
(1) For a given $\alpha \in \bF_q$, let $A_m = A_{m, \alpha}$ be the number of monic squarefree polynomials $h$ of degree $m$ satisfying $h(\alpha) \neq 0$.
Then for $m\geq1$ we have
\begin{align*}
    A_m=\frac{(q-1)(q^m-(-1)^m)}{q+1}=(q-1)\sum_{i=0}^{m-1}(-1)^{m-1-i}q^i.
\end{align*}
(2) For a given $\alpha \neq \beta$, let $B_m = B_{m, \alpha, \beta}$ be the number of monic squarefree polynomials $h$ of degree $m$ such that $h(\alpha)h(\beta) \neq 0$.
Then, 
\begin{align*}
    B_{2m} &=  1 + \frac{q(q-1)(q^{2m}-1)}{(q+1)^2} - \frac{2m(q-1)}{(q+1)}, \qquad 
    B_{2m + 1} =  q-2 + \frac{q^2(q-1)(q^{2m}-1)}{(q+1)^2} + \frac{2m(q-1)}{(q+1)}.
\end{align*}
(3) For a given monic irreducible quadratic $u \in \bF_q[x]$, let $C_m = C_{m, u}$ be the number of monic squarefree polynomials $h$ of degree $m$ such that $(h, u) = 1$.
Then,
\begin{align*}
    C_{2m} &= \frac{q^{2m+2} - q^{2m+1} + (-1)^m(q+1) }{q^2 + 1}, \qquad 
    C_{2m+1} = \frac{q^{2m+3} - q^{2m+2} + (-1)^mq(q+1) }{q^2 + 1}.
\end{align*}
\end{corollary}
\begin{proof}
Among all monic squarefree polynomials of degree $m \geq 2$, the number of monic squarefree polynomials of degree $m$ vanishing at $\alpha$ is $A_{m-1}$ since such a polynomial can be written uniquely as $h(x)=(x-\alpha)g(x)$.
Hence by Lemma \ref{lem: num sqf poly}, we have $A_1 = q-1$ and
\begin{align*}
    A_m + A_{m-1} = q^m - q^{m-1}, \qquad \textrm{for } m \geq 2.
\end{align*}

The number of monic squarefree polynomials of degree $m$ vanishing at both $\alpha$ and $\beta$ is $B_{m-2}$, since such polynomial can be written by $h(x) = (x - \alpha)(x - \beta)g(x)$.
Hence, by Lemma \ref{lem: num sqf poly}, we have $B_0 = 1, B_1 = q-2$ and 
\begin{align*}
    B_m = (q^m - q^{m-1}) - 2A_{m-1} + B_{m-2} \qquad  \textrm{for } m\geq 2.
\end{align*}

If a monic squarefree polynomial $h$ is divided by a monic irreducible quadratic polynomial $u$, then $h(x) = u(x)v(x)$ for a monic squarefree polynomial $v(x)$ of degree $m-2$.
Hence, $C_0 = 1$, $C_1 = a$ and 
\begin{align*}
    C_m = (q^m - q^{m-1}) - C_{m-2} \qquad  \textrm{ for } m\geq 2.
\end{align*}

We can deduce the closed formulas by solving the recurrence relations.
\end{proof}

For $\bfz = (z_2, \cdots, z_{2g+1}) \in \bA^{2g}(\bF_q)$, write
\begin{align*}
    f_{\bfz}(x) = \prod_{i=1}^k (x - \alpha_i)^{n_i}, \quad \alpha_i \in \overline{\bF_q} \textrm{ distinct,} \quad n_1 \geq n_2 \geq \cdots \geq n_k.
\end{align*}
Following the terminology in section \ref{subsec: naive red}, we define the geometric multiplicity partition of $f_{\bfz}$ by 
\begin{align*}
    F(f_{\bfz}) := (n_1, \cdots, n_k) \vdash \deg f_{\bfz} = 2g+1.
\end{align*}
Since $f_{\lambda \cdot \bfa}(x) = \lambda^{2(2g+1)}f_{\bfa}(\lambda^{-2}x)$, the multiplicity partition is well-defined for $\bfz \in \cP(w_g)(\bF_q)^\ast$.
For $\bfa \in \cP(w_g)(K)$, we define $F_{\fp}(f_{\bfa})$ as $F(f_{\bfz})$ where $\psi_{\fp}(\bfa) = \bfz$, and $(2g+1)$ if $\psi_{\fp}(\bfa) = [0]$.

For a partition $\lambda \vdash 2g+1$, let
\begin{align*}
    S_{\lambda} := \lcrc{\bfz = (a_2, \cdots, a_{2g+1}) \in \bA^{2g}(\bF_q) \setminus \lcrc{0} : F(f_{\bfz}) = \lambda}.
\end{align*}

\begin{proposition} \label{prop: affine local condition counting}
Let $p > 2g + 1$ be a prime. Then, 
\begin{align*}
    \#S_{(1^{2g+1})} &= q^{2g} - q^{2g-1}, \\
    \#S_{(2,1^{2g-1})}
    &= A_{2g-1}
    = (q-1)(q^{2g-2}-q^{2g-3}+\cdots-q+1),\\
    \#S_{(3,1^{2g-2})}
    &= A_{2g-2}
    = (q-1)(q^{2g-3}-q^{2g-4}+\cdots-1),\\
    \#S_{(2,2,1^{2g-3})}
    &= \frac{q-1}{2}\bigl(B_{2g-3}+C_{2g-3}\bigr),
\end{align*}
where $B_n$ and $C_n$ are given by Corollary \ref{cor: sqf poly variants}.
\end{proposition}
\begin{proof}
For $\lambda = (1^{2g+1})$, $f_{\bfz}(x)$ is a monic squarefree polynomial over $\bF_q$ whose coefficient of $x^{2g}$ is zero.
The number of monic squarefree polynomials of degree $2g+1$ over $\bF_q$ is $q^{2g+1} - q^{2g}$ by Lemma \ref{lem: num sqf poly}.
Note that the translation $x\mapsto x-t/(2g+1)$ annihilates the coefficient of $x^{2g}$ and gives a $q$-to-$1$ map from the set of monic squarefree polynomials to the set of monic squarefree polynomials with $a_1=0$. Since $p>2g+1$, this translation is well-defined and invertible. Hence, $\#S_{(1^{2g+1})}=q^{2g}-q^{2g-1}$.

For $\lambda=(2,1^{2g-1})$, the polynomial $f_{\bfz}(x)$ has the form $f_{\bfz}(x)=(x-\alpha)^2h(x),$ where $h$ is a monic squarefree polynomial of degree $2g-1$ satisfying $h(\alpha)\neq 0$. By Corollary \ref{cor: sqf poly variants} (1), the number of such polynomials with a given $\alpha$ is $A_{2g-1}$.
Also, there are $q$-choices of $\alpha$, and the translation argument cancels the $q$-choices. 
So, $\#S_{2, 1^{2g-1}} = A_{2g-1}$.

For $\lambda=(3,1^{2g-2})$, the polynomial $f_{\bfz}(x)$ has the form $f_{\bfz}(x)=(x-\alpha)^3h(x),$ where $h$ is a monic squarefree polynomial of degree $2g-2$ satisfying $h(\alpha)\neq 0$. 
Again, by Corollary \ref{cor: sqf poly variants} (1), the number of such polynomials is $A_{2g-2}$.

For $\lambda=(2,2,1^{2g-3})$, the polynomial $f_{\bfz}(x)$ has the form $f_{\bfz}(x)=u(x)^2v(x),$ where $u$ is a monic squarefree quadratic polynomial, $v$ is a monic squarefree polynomial of degree $2g-3$, and $(u,v)=1$.
There are two cases: either $u$ splits over $\mathbb{F}_q$ or $u$ is irreducible over $\mathbb{F}_q$.

If $u$ splits, say $u(x)=(x-\alpha)(x-\beta)$ with $\alpha\neq\beta$, then the number
of such $v$ is $B_{2g-3}$ by Corollary \ref{cor: sqf poly variants} (2). The number of
monic split squarefree quadratics is
\[
\binom{q}{2}=\frac{q(q-1)}{2}.
\]

If $u$ is irreducible, then the number of such $v$ is $C_{2g-3}$ by
Corollary \ref{cor: sqf poly variants} (3). The number of monic irreducible quadratics
over $\bF_q$ is 
\[
q^2 - \lbrb{ \binom{q}{2} + q}= \frac{q(q-1)}{2}.
\]
Therefore, the total number of monic degree $2g+1$ polynomials with multiplicity
partition $(2,2,1^{2g-3})$ is
\[
\frac{q(q-1)}{2}\bigl(B_{2g-3}+C_{2g-3}\bigr).
\]
Removing the $x^{2g}$-coefficient by translation divides the count by $q$, so we obtain the last one.
\end{proof}

For the weighted projective analogue of the affine set $S_{\lambda}$, we define
\begin{align*}
    \cS_{\lambda} := \lcrc{\bfz \in \cP(w_g)(\bF_q) : F(f_{\bfz}) = \lambda}.
\end{align*}
We choose a Haar measure $\mu$ on $\cO_{K, \fp}^{2g}$ normalized by $\mu(\cO_{K, \fp}^{2g}) = 1$.
Recall that
\begin{align*}
    U_{g, \fp} = \cO_{K, \fp}^{2g} \setminus \prod_{i=2}^{2g+1} \pi_{\fp}^{2i} \cO_{K, \fp}.
\end{align*}
Thus $\mu(U_{g, \fp})=1-q^{-|w_g|}$.
For a measurable set $A\subseteq \cO_{K, \fp}^{2g}$, define
\begin{align*}
    \mu_g(A)
    :=\frac{\mu(A\cap U_{g, \fp})}{1-q^{-|w_g|}}.
\end{align*}
We note that $\mu_g$ gives the local density appearing in Proposition \ref{prop: counting P loc cond}.
Indeed, for $\bfz\in\cP(w_g)(\bF_q)$, let $A_{\fp,\bfz}=\Omega_{\fp,\bfz}^{\aff}\cap U_{g,\fp}$ as in its proof. Then
\begin{align*}
    \mu_g(A_{\fp,\bfz})
    =\frac{q-1}{\#\mu_{d(\bfz)}(\bF_q)}
    \frac{q^{-2g}}{1-q^{-|w_g|}}.
\end{align*}
Hence the leading term in Proposition \ref{prop: counting P loc cond} is precisely $\kappa\mu_g(A_{\fp,\bfz})X^{|w_g|}$.

\begin{lemma} \label{lem: regular measure}
Let $\mathcal L_{\overline f} \subset \cO_{K, \fp}^{2g}$ denote the set of lifts of $\overline f$ in the trace-zero coefficient space.
Then
\begin{align*}
\frac{\mu\bigl(
        \{f\in\mathcal L_{\overline f}:
        f\text{ is separable and }
        \mathcal C_f\text{ is not regular}\}
    \bigr)}{\mu(\cL_{\overline{f}})}
    =
    O_g(q^{-1}),
\end{align*}
uniformly in $\overline f$.
\end{lemma}
\begin{proof}
Fix a lift $f_0\in\mathcal L_{\overline f}$.
Every $f \in \cL_{\overline{f}}$ can be written uniquely as $f_0 + \pi_{\fp} \phi$ for $\phi \in \cO_{K, \fp}[x]$ satisfying $\deg \phi \leq 2g - 1$.
Let
\begin{align*}
    \overline f(x)
    =
    \prod_{i=1}^r h_i(x)^{e_i}
    \in \kappa(\fp)[x], \qquad d_i := \deg h_i.
\end{align*}
For each $i$, choose a root $\overline c_i$ of $h_i$. 
Let $K_{\fp,i}/K_\fp$ be the unramified extension with residue field $k_i=\kappa(\fp)(\overline c_i),$ and choose a lift $c_i\in\mathcal O_{K_{\fp,i}}$ of $\overline c_i$.
Since $\overline{f}(\overline{c}_i) = 0$, we have $f(c_i) \in \pi_{\fp} \cO_{K_{\fp, i}}$ and
\begin{align*}
\pi_{\fp}^{-1}f(c_i)
\equiv
\pi_{\fp}^{-1}f_0(c_i)
+\overline\phi(\overline c_i)
\pmod{\pi_{\fp}},
\end{align*}
where the right-hand side is regarded as an element of $k_i$.
Put
\begin{align*}
    D:=\sum_{e_i\geq2}d_i.
\end{align*}
Since
\begin{align*}
    2D
    \leq
    \sum_{e_i\geq2}e_id_i
    \leq 2g+1,
\end{align*}
we have $D\leq g$. 
By the Chinese remainder theorem, the evaluation map gives an isomorphism
\begin{align*}
    k[x]_{<D}
    &\overset{\sim}{\longrightarrow} \left.k[x]\middle/\prod_{e_i \geq 2} h_i(x) \right.
    \overset{\sim}{\longrightarrow} 
    \prod_{e_i\geq2} k[x]/h_i(x)
    \overset{\sim}{\longrightarrow} \prod_{e_i \geq 2} k_i.
\end{align*}
Hence, the evaluation map from $k[x]_{\leq 2g-1}$ to $\prod_{e_i \geq 2}k_i$ is surjective.
Since this is a surjective $k$-linear map, all its fibers have the same cardinality.
Hence, as $f$ varies in $\mathcal L_{\overline f}$, the tuple
\begin{align*}
    \bigl(\pi^{-1}f(c_i)\bmod\pi\bigr)_{e_i\geq2}
\end{align*}
is uniformly distributed in $\prod_{e_i\geq2}k_i$.

By Proposition~\ref{prop:regular-lifts}, a separable lift is nonregular precisely when at least one of these coordinates vanishes. Therefore,
\begin{align*}
\frac{\mu\bigl(
        \{f\in\mathcal L_{\overline f}:
        f\text{ is separable and }\mathcal C_f
        \text{ is not regular}\}
    \bigr)}
    {\mu(\mathcal L_{\overline f})}
&\leq
\sum_{e_i\geq2}\frac{1}{|k_i|} =
\sum_{e_i\geq2}q^{-d_i}
\leq gq^{-1}.
\end{align*}
This proves the result.
\end{proof}

\begin{proposition} \label{prop: affine partition bound}
For a partition $\lambda\vdash n$ with $\lambda\neq(n)$, let $S_\lambda$ be as above. Then
\begin{align*}
    \#S_\lambda
    &=
    \sum_{\bfz\in\cS_\lambda}
    \frac{q-1}{\#\mu_{d(\bfz)}(\bF_q)} =
    q^{n-1-d(\lambda)}
    \bigl(1+O_g(q^{-1})\bigr).
\end{align*}
For $\lambda=(n)$, one has $S_{(n)}=\varnothing$.
Also, 
\begin{align*}
    \mu\{
        \bfa \in \cO_{K, \fp}^{n-1} :
        \overline f_{\bfa}\in S_\lambda
    \}
    =
    q^{-d(\lambda)}
    +O_g(q^{-d(\lambda)-1})
\end{align*}
for $\lambda\neq(n)$.
\end{proposition}

\begin{proof}
By (\ref{eqn: cP Fq to A Fq}), we have the first equality.
Let $r_m$ be the number of parts of $\lambda$ equal to $m$, and put
\begin{align*}
    \ell(\lambda):=\sum_{m=1}^n r_m.
\end{align*}
Before imposing the trace-zero condition, every monic polynomial $\overline f$ of degree $n$ with geometric multiplicity partition $\lambda$ admits a unique factorization 
\begin{align*}
    \overline f(x)=\prod_{m=1}^n H_m(x)^m,
\end{align*}
where each $H_m$ is monic and squarefree, the $H_m$ are pairwise coprime, and $\deg H_m=r_m.$
We set $H_m=1$ when $r_m=0$.

There are $q^{\ell(\lambda)}$ tuples of monic polynomials $(H_m)_{m=1}^n$ satisfying only the degree conditions $\deg H_m=r_m$. 
We claim that the number of such tuples for which either some $H_m$ is not squarefree or some pair $H_m,H_\ell$ is not coprime is $O_g(q^{\ell(\lambda)-1}).$
Indeed, for fixed $m$, if $H_m$ is not squarefree, then $P^2\mid H_m$ for some monic irreducible polynomial $P$. 
For each $d\geq1$, there are $O(q^d)$ monic irreducible polynomials of degree $d$, while for a fixed such $P$ there are at most $q^{r_m-2d}$ possible monic polynomials $H_m$ divisible by $P^2$. Hence
\begin{align*}
    \#\{H_m:\deg H_m=r_m,\ H_m\text{ not squarefree}\}
    &\ll_g
    \sum_{d\geq1}q^d q^{r_m-2d} \ll_g q^{r_m-1}.
\end{align*}
Allowing the remaining $H_j$ for $j \neq m$ to vary freely therefore gives $O_g(q^{\ell(\lambda)-1})$ such tuples.

Similarly, for fixed $m\neq\ell$, if $H_m$ and $H_\ell$ are not coprime, then some monic irreducible polynomial $P$ divides both. 
For a fixed monic irreducible $P$ of degree $d$, there are at most $q^{r_m-d}$ choices of $H_m$ divisible by $P$ and at most $q^{r_\ell-d}$ choices of $H_\ell$ divisible by $P$. Therefore
\begin{align*}
    \#\{(H_m,H_\ell):(H_m,H_\ell)\neq1\}
    &\ll_g
    \sum_{1 \leq d\leq \min(r_m, r_{\ell}) }q^d q^{r_m-d}q^{r_\ell-d} \ll_g q^{r_m+r_\ell-1}.
\end{align*}
After allowing the remaining $H_j$ for $j \neq m, \ell$ to vary freely, this again contributes $O_g(q^{\ell(\lambda)-1})$ tuples. 
Since there are only $O_g(1)$ choices of $m$ and of pairs $(m,\ell)$, the total number of bad tuples is $O_g(q^{\ell(\lambda)-1})$.
Consequently, the number of monic degree-$n$ polynomials with geometric multiplicity partition $\lambda$, before imposing the trace-zero condition, is $q^{\ell(\lambda)}\bigl(1+O_g(q^{-1})\bigr).$

As in the proof of Proposition \ref{prop: affine local condition counting}, since $p>n$, translation by elements of $k$ gives orbits of size $q$ on the set of monic degree-$n$ polynomials, and each orbit contains a unique polynomial whose $x^{n-1}$-coefficient is zero. 
Since translation preserves the geometric multiplicity partition, dividing the unrestricted count by $q$ and using $\ell(\lambda) = n - d(\lambda)$ gives the result.
Finally, 
\begin{align*}
\mu\{
\bfa\in \cO_{K, \fp}^{n-1}:
\overline f_{\bfa}\in S_\lambda
\}
&=
\frac{\#S_\lambda}{q^{n-1}}=
q^{-d(\lambda)}
+O_g(q^{-d(\lambda)-1}).
\end{align*}
So we have the results.
\end{proof}

\subsection{Proof of Theorem \ref{mainthm: probability}}

We recall that $a, t, u$ are the abelian, toric, and unipotent ranks of the Jacobian of genus $g$ hyperelliptic curves.
Also, for $\tau = (a, t, u)$, define
\begin{align*}
    P_{\fp}(\tau) := \mu_g \lcrc{\bfa \in U_{g, \fp} :  \Delta(f_{\bfa}) \neq 0, \tau(f_{\bfa}) = \tau }.
\end{align*}

\begin{theorem}
\label{thm:general-leading-probabilities}
Let $K$ be a number field and let $\fp$ be a prime of $K$ with residue field $\bF_q$ of characteristic $p > 2g+1$.
For nonnegative integers $a,t,u$ with $a+t+u=g$, put
\[
    M_g(a,t,u)
    =
    \#\{
        \lambda\vdash2g+1:
        \tau_0(\lambda)=(a,t,u)
    \}.
\]
Then, 
\begin{equation}
\label{eq:general-leading-probabilities}
    P_{\fp}(\tau)
    =
    \frac{M_g(a,t,u)}{q^{t+2u}}
    +O_g(q^{-t-2u-1}).
\end{equation}
\end{theorem}

\begin{proof}
We first treat the exceptional partition $\lambda=(2g+1)$.
Although $S_{(2g+1)}=\varnothing$, the corresponding affine residue class is the zero coefficient vector, namely $\{\bfa\in\cO_{K,\fp}^{2g}: \overline f_{\bfa}(x)=x^{2g+1}\}$, whose Haar measure is $q^{-2g}$.
Also, $\tau_0((2g+1)) = (0, 0, g)$ and $d((2g+1)) = 2g$.
By Lemma \ref{lem: regular measure}, all but a relative proportion $O_g(q^{-1})$ of the lifts in the set are regular, and  Proposition~\ref{prop:regular-lifts} shows that every regular lift satisfies $\tau(f) = (0,0,g)$.
Thus
\begin{align*}
    \mu\{\bfa\in \cO_{K, \fp}^{n-1} : \overline{f}_{\bfa}(x) = x^n, \,  \tau(f_{\bfa})=(0,0,g)\}
    =
    q^{-2g}+O_g(q^{-2g-1}).
\end{align*}

Let $\lambda \neq (2g+1)$.
We first work with $\mu$ on the full coefficient space and fix $\tau=(a,t,u)$.
Put $c=t+2u$. By Proposition \ref{lem:rank-comparison}, every separable lift of a polynomial in $S_\lambda$ satisfies
\begin{align*}
    t(f)+2u(f)\leq d(\lambda).
\end{align*}
Hence a stratum with $d(\lambda)<c$ does not contribute to the event $\tau(f)=\tau$.

Suppose now that $d(\lambda)=c$.
By Lemma \ref{lem: regular measure}, all but a relative proportion $O_g(q^{-1})$ of the lifts above $S_\lambda$ are regular, and by Proposition \ref{prop:regular-lifts} every such regular lift has $\tau(f)=\tau_0(\lambda).$
Together with Proposition \ref{prop: affine partition bound}, this gives
\begin{align*}
    \mu\lcrc{
        \bfa\in\cO_{K,\fp}^{n-1}:
        \overline f_{\bfa}\in S_\lambda,\,
        \tau(f_{\bfa})=\tau
    }
    =
    \begin{cases}
        q^{-c}+O_g(q^{-c-1}),
            &\tau_0(\lambda)=\tau,\\
        O_g(q^{-c-1}),
            &\tau_0(\lambda)\neq\tau.
    \end{cases}
\end{align*}
The total measure of the strata with $d(\lambda)>t + 2u$ is $O_g(q^{-c-1})$.
There are only finitely many partitions of $n$, depending on $g$, so summing proves \eqref{eq:general-leading-probabilities} for $\mu$.
For every measurable event $A$,
\[
    |\mu_g(A)-\mu(A)|
    \leq\frac{q^{-|w_g|}}{1-q^{-|w_g|}}.
\]
Since $c\leq2g$ and $|w_g|>2g+1$, this change of normalization
is absorbed by the stated error.
It gives \eqref{eq:general-leading-probabilities}.
\end{proof}



We recall that the probability of $C \in \cH_{g, 1}(K)$ satisfies $(a, t, u)$ at $\fp$ is defined in \eqref{eqn: property}.

\begin{proposition}
\label{prop:local-global-probability}
Let $K$ be a number field and let $\mathfrak p$ be a prime of $K$ with residue characteristic $p>2g+1$.
Then, for every triple of nonnegative integers $(a,t,u)$ with $a+t+u=g$,
\begin{align*}
    \mathrm{Pr}_{\mathfrak p}(a,t,u)
    = P_{\fp}(a, t, u)
\end{align*}
\end{proposition}

\begin{proof}
Let $\pi_{\fp}$ be a uniformizer of $\cO_{K,\fp}$ and put $D=[K:\bQ]$.
For a point $\bfa\in\cH_{g,1}(K)$, choose a weighted primitive representative $\bfa_{\fp}\in U_{g,\fp}$ of its image in $\cP(w_g)(K_{\fp})$.

We first record the local-to-global counting statement that will be used below.
Let $A\subset U_{g,\fp}$ be a weighted $\cO_{K,\fp}^{\times}$-invariant subset whose membership is determined modulo $\pi_{\fp}^N$ for some $N\geq1$.
By the same argument as in Proposition \ref{prop: counting P loc cond}, using \cite[Theorem 6.18]{CJP} with congruences modulo $\fp^N$, we have
\begin{align*}
    \#\lcrc{\bfa\in\cP(w_g)(K)(X):\bfa_{\fp}\in A}
    =\kappa\mu_g(A)X^{|w_g|}
    +O_{K,g,\fp,A}\lbrb{X^{|w_g|-\frac4D}\log X}.
\end{align*}
Indeed, if $\Omega_A\subset\cP(w_g)(K_{\fp})$ denotes the corresponding projective local condition, then
\begin{align*}
    \Omega_A^{\aff}\cap\cO_{K,\fp}^{2g}
    =\bigsqcup_{j\geq0}\pi_{\fp}^j*_{w_g}A.
\end{align*}
Hence
\begin{align*}
    \mu\lbrb{\Omega_A^{\aff}\cap\cO_{K,\fp}^{2g}}
    =\sum_{j\geq0}q^{-j|w_g|}\mu(A)
    =\frac{\mu(A)}{1-q^{-|w_g|}}
    =\mu_g(A).
\end{align*}
By Lemma \ref{lem: Hg1 counting}, the number of points in $\cP(w_g)(K)(X)$ with zero discriminant is $O_{K,g}(X^{|w_g|-(4g+2)})$.
Since $4g+2>4/D$, removing these points does not affect the preceding estimate. Therefore,
\begin{align*}
    \#\lcrc{\bfa\in\cH_{g,1}(K)(X):\bfa_{\fp}\in A}
    =\kappa\mu_g(A)X^{|w_g|}
    +O_{K,g,\fp,A}\lbrb{X^{|w_g|-\frac4D}\log X}.
\end{align*}
Together with Lemma \ref{lem: Hg1 counting}, this gives
\begin{align} \label{eqn: Hg1 local density}
    \lim_{X\to\infty}
    \frac{\#\lcrc{\bfa\in\cH_{g,1}(K)(X):\bfa_{\fp}\in A}}
    {\#\cH_{g,1}(K)(X)}
    =\mu_g(A).
\end{align}

Now fix $\tau=(a,t,u)$ with $a+t+u=g$.
For $B\geq0$, put
\begin{align*}
    A_{\tau,B}
    &:=
    \lcrc{\bfa\in U_{g,\fp}:
    \ord_{\fp}\lbrb{\Delta(f_{\bfa})}\leq B,\,
    \tau(f_{\bfa})=\tau},\\
    E_B
    &:=
    \lcrc{\bfa\in U_{g,\fp}:
    \ord_{\fp}\lbrb{\Delta(f_{\bfa})}>B}.
\end{align*}
Since the rank triple $\tau$ is locally constant on the locus $\Delta(f_{\bfa})\neq0$ (cf. \cite[Theorem 1.29, \S 19]{DDMM}), and the set
\begin{align*}
    \lcrc{\bfa\in U_{g,\fp}:
    \ord_{\fp}\lbrb{\Delta(f_{\bfa})}\leq B}
\end{align*}
is compact, there exists an integer $N=N(B,\tau)$ such that membership in $A_{\tau,B}$ is determined modulo $\pi_{\fp}^N$.
Moreover, $A_{\tau,B}$ is invariant under weighted $\cO_{K,\fp}^{\times}$-scaling.
Similarly, $E_B$ is determined by the congruence $\Delta(f_{\bfa})\equiv0\pmod{\pi_{\fp}^{B+1}}$ and is invariant under weighted unit scaling.
Thus \eqref{eqn: Hg1 local density} applies to both $A_{\tau,B}$ and $E_B$.

For every $B$, the event $\tau(C_{\bfa}/K_{\fp})=\tau$ differs from the truncated event $\bfa_{\fp}\in A_{\tau,B}$ only when $\bfa_{\fp}\in E_B$.
Therefore,
\begin{align*}
    &\#\lcrc{\bfa\in\cH_{g,1}(K)(X):\bfa_{\fp}\in A_{\tau,B}}\\
    &\qquad\leq
    \#\lcrc{\bfa\in\cH_{g,1}(K)(X):\tau(C_{\bfa}/K_{\fp})=\tau}\\
    &\qquad\leq
    \#\lcrc{\bfa\in\cH_{g,1}(K)(X):\bfa_{\fp}\in A_{\tau,B}}
    +\#\lcrc{\bfa\in\cH_{g,1}(K)(X):\bfa_{\fp}\in E_B}.
\end{align*}
Dividing by $\#\cH_{g,1}(K)(X)$ and letting $X\to\infty$, we obtain
\begin{align*}
    \mu_g(A_{\tau,B})
    &\leq
    \liminf_{X\to\infty}
    \frac{\#\lcrc{\bfa\in\cH_{g,1}(K)(X):
    \tau(C_{\bfa}/K_{\fp})=\tau}}
    {\#\cH_{g,1}(K)(X)}\\
    &\leq
    \limsup_{X\to\infty}
    \frac{\#\lcrc{\bfa\in\cH_{g,1}(K)(X):
    \tau(C_{\bfa}/K_{\fp})=\tau}}
    {\#\cH_{g,1}(K)(X)}
    \leq\mu_g(A_{\tau,B})+\mu_g(E_B)
\end{align*}
by \eqref{eqn: Hg1 local density}.
The sets $E_B$ decrease as $B\to\infty$, and
\begin{align*}
    \bigcap_{B\geq0}E_B
    =\lcrc{\bfa\in U_{g,\fp}:\Delta(f_{\bfa})=0}.
\end{align*}
Since the discriminant-zero locus has Haar measure zero, we have $\mu_g(E_B)\to0$.
On the other hand,
\begin{align*}
    A_{\tau,B}\nearrow
    \lcrc{\bfa\in U_{g,\fp}:\Delta(f_{\bfa})\neq0,\,
    \tau(f_{\bfa})=\tau},
\end{align*}
and hence $\mu_g(A_{\tau,B})\to P_{\fp}(a,t,u)$.
Letting $B\to\infty$ in the preceding inequalities yields $\mathrm{Pr}_{\fp}(a,t,u)=P_{\fp}(a,t,u),$
as claimed.
\end{proof}

\begin{proof}[Proof of Theorem~\ref{mainthm: probability}]
By Theorem~\ref{thm:general-leading-probabilities} and Proposition~\ref{prop:local-global-probability}, we have
\begin{align*}
    \mathrm{Pr}_{\fp}(a,t,u)
    =P_{\fp}(a,t,u)
    =\frac{M_g(a,t,u)}{q^{t+2u}}+O_g(q^{-t-2u-1})
\end{align*}
for every triple $(a,t,u)$ with $a+t+u=g$.

It remains to refine the estimate for $(g,0,0)$. By the definition of $M_g$, the partition $(2,1^{2g-1})$ is the unique partition corresponding to $(g-1,1,0)$. Hence $M_g(g-1,1,0)=1$, and
\begin{align*}
    \mathrm{Pr}_{\fp}(g-1,1,0)=\frac{1}{q}+O_g(q^{-2}).
\end{align*}
On the other hand, if $(a,t,u)\neq(g,0,0),(g-1,1,0)$, then $t+2u\geq2$. Since there are only finitely many triples $(a,t,u)$ with $a+t+u=g$, the first assertion gives
\begin{align*}
    \sum_{\substack{a+t+u=g\\
    (a,t,u)\neq(g,0,0),(g-1,1,0)}}
    \mathrm{Pr}_{\fp}(a,t,u)=O_g(q^{-2}).
\end{align*}
Since the probabilities of all rank triples sum to $1$, we obtain
\begin{align*}
    \mathrm{Pr}_{\fp}(g,0,0)
    &=1-\mathrm{Pr}_{\fp}(g-1,1,0)+O_g(q^{-2}) =1-\frac{1}{q}+O_g(q^{-2}).
\end{align*}
This completes the proof.
\end{proof}

\section{Average rank} \label{sec: AR}

\subsection{Moments}

For $\bfz\in\cP(w_g)(\bF_q)$, let $C_{\bfz}$ denote the possibly singular projective curve over $\bF_q$ defined by $y^2=f_{\bfz}(x)$.
For every $m\geq 1$, we define
\begin{align*}
a_{\fp^m}(C_{\bfz})
:=q^m+1-\#C_{\bfz}(\bF_{q^m}).    
\end{align*}
This is a point-counting convention, distinct from the local Euler power trace of a smooth curve over $K$.
The two agree under the hypotheses of Proposition \ref{prop: semistable simple double reduction} (3).

We recall that  
\begin{align*}
\cS_{(2,1^{2g-1})}
:=
\left\{
\bfz\in \mathcal P(w_g)(\bF_q): F(f_{\bfz})=(2,1^{2g-1})
\right\}.    
\end{align*}
The goal of this section is to give an estimate of
\begin{align*}
    \sum_{\substack{\bfz \in \cP(w_g)(\bF_q) \\ \Delta(\bfz) \not\equiv 0}} \frac{q-1}{\# \mu_{d(\bfz)}(\bF_q)} a_{\fp^m}(C_{\bfz}), \qquad 
    \sum_{\bfz \in \cS_{(2, 1^{2g-1})}} \frac{q-1}{\# \mu_{d(\bfz)}(\bF_q)} a_{\fp^m}(C_{\bfz}),
\end{align*}
for $m = 1, 2$.
Let $\chi_{q^m}$ be a quadratic character modulo $q^m$ for $m = 1, 2$. 
We recall that 
\begin{align} \label{eqn: Frob relation}
a_{\fp^m}(C_{\bfz})=-\sum_{x\in\bF_{q^m}}\chi_{q^m} \lbrb{f_{\bfz}(x)}.
\end{align}

\begin{proposition} \label{prop: firstmoment zero}
Assume that $q$ is odd. Then
\begin{align*}
\sum_{\substack{\bfz \in \cP(w_g)(\bF_q) \\ \Delta(\bfz) \not\equiv 0}}
\frac{q-1}{\#\mu_{d(\bfz)}(\bF_q)} a_{\fp}(C_{\bfz})=0.
\end{align*}
\end{proposition}

\begin{proof}
For $\eta \in \bF_q^\times$, we define
\begin{align*}
T_{\eta}(\bfz):= (\eta^{-2}z_2,\eta^{-3}z_3,\dots,\eta^{-(2g+1)}z_{2g+1}).
\end{align*}
Then, $T_{\eta}$ induces a bijection on $\cP(w_g)(\bF_q)$, and $f_{T_{\eta}(\bfz)}(x) = \eta^{-(2g+1)}f_{\bfz}(\eta x)$.
By (\ref{eqn: Frob relation}), for a non-square $\eta \in \bF_q^\times$, we have
\begin{align*}
a_{\fp}(C_{T_\eta(\bfz)})
&= -\sum_{x\in\bF_q}\chi_q \lbrb{\eta^{-(2g+1)}f_{\bfz}(\eta x)}
= -\chi_q(\eta)\sum_{x\in\bF_q}\chi_q \lbrb{f_{\bfz}(\eta x)}
= \chi_q(\eta)a_{\fp}(C_{\bfz}) = -a_{\fp}(C_{\bfz}).
\end{align*}
Because $T_{\eta}$ does not change which coordinate vanishes, we have $d(T_{\eta}(\bfz)) = d(\bfz)$.
Thus, for any subset $\mathcal T \subseteq \cP(w_g)(\bF_q)$ preserved by $T_\eta$, we obtain
\begin{align*}
\sum_{\bfz \in \mathcal T} \frac{q-1}{\#\mu_{d(\bfz)}(\bF_q)} a_{\fp}(C_{\bfz})
&= \sum_{\bfz \in \mathcal T} \frac{q-1}{\#\mu_{d(T_\eta(\bfz))}(\bF_q)} a_{\fp}(C_{T_\eta(\bfz)}) = -\sum_{\bfz \in \mathcal T} \frac{q-1}{\#\mu_{d(\bfz)}(\bF_q)} a_{\fp}(C_{\bfz}),
\end{align*}
so the sum is zero.
Since $T_\eta$ preserves the discriminant-zero locus, taking $\mathcal T=\{\bfz\in\cP(w_g)(\bF_q):\Delta(\bfz)\neq 0\}$ proves the proposition.
\end{proof}

\begin{proposition}\label{prop: t=1 weighted vanishes}
Assume that $q$ is odd. Then
\begin{align*}
\sum_{\bfz \in \cS_{(2,1^{2g-1})}}
\frac{q-1}{\#\mu_{d(\bfz)}(\bF_q)} a_{\fp}(C_{\bfz})=0.
\end{align*}
\end{proposition}

\begin{proof}
Since $f_{T_\eta(\bfz)}(x)=\eta^{-(2g+1)}f_{\bfz}(\eta x)$, the map $T_\eta$ preserves the geometric root multiplicity partition.
In particular, it induces a bijection on $\cS_{(2,1^{2g-1})}$.
The result follows from the proof of Proposition \ref{prop: firstmoment zero}.
\end{proof}

\begin{proposition}\label{prop: secondmoment}
Assume that \(q\) is odd. Then
\begin{align*}
    \sum_{\substack{\bfz \in \cP(w_g)(\bF_q) \\ \Delta(\bfz) \not\equiv 0 \pmod{\fp}}} \frac{q-1}{\# \mu_{d(\bfz)}(\bF_q)} a_{\fp^2}(C_{\bfz})
    = -q^{2g+1} + O(q^{2g}).
\end{align*}
\end{proposition}

\begin{proof}
By (\ref{eqn: cP Fq to A Fq}), we have 
\begin{align*}
-\sum_{\bfz \in \cP(w_g)(\bF_q)}
\frac{q-1}{\#\mu_{d(\bfz)}(\bF_q)}  \sum_{x \in \bF_{q^2}} \chi_{q^2} (f_{\bfz}(x))
&=
-\sum_{x \in \bF_{q^2}}\sum_{\bfz \in \bA^{2g}(\bF_q)} \chi_{q^2} (f_{\bfz}(x))
+ \sum_{x \in \bF_{q^2}} \chi_{q^2}(f_{0}(x)) \\    
&=
-\sum_{x \in \bF_{q^2}}\sum_{\bfz \in \bA^{2g}(\bF_q)} \chi_{q^2} (f_{\bfz}(x))
+ \sum_{x \in \bF_{q^2}} \chi_{q^2}(x^{2g+1}) \\
&=
-\sum_{x \in \bF_{q^2}}\sum_{\bfz \in \bA^{2g}(\bF_q)} \chi_{q^2} (f_{\bfz}(x)).
\end{align*}
We evaluate the inner sum according to whether \(x\in \bF_q\).
If \(x\notin \bF_q\), then $\{1,x\}$ form an \(\bF_q\)-basis of \(\bF_{q^2}\).
Fix \(z_2,\dots,z_{2g-1}\in \bF_q\). As \((z_{2g},z_{2g+1})\) varies in \(\bF_q^2\), the value
\[
    f_{\bfz}(x)
    =
    \bigl(x^{2g+1}+z_2x^{2g-1}+\cdots+z_{2g-1}x^2\bigr)
    + z_{2g}x + z_{2g+1}
\]
runs through all of \(\bF_{q^2}\). Therefore if $x \not\in \bF_q$,
\[
\sum_{\bfz \in \bA^{2g}(\bF_q) } \chi_{q^2}\bigl(f_{\bfz}(x)\bigr)=0.
\]

If \(x\in \bF_q\), then \(f_{\bfz}(x)\in \bF_q\). Fix \(z_2,\dots,z_{2g}\in \bF_q\). As
\(z_{2g+1}\) varies in \(\bF_q\), the value \(f_{\bfz}(x)\) runs through all of \(\bF_q\).
Since the restriction of \(\chi_{q^2}\) to \(\bF_q^\times\) is trivial, we obtain
\[
\sum_{z_{2g+1}\in \bF_q}\chi_{q^2}\bigl(f_{\bfz}(x)\bigr)
=
\sum_{y\in \bF_q}\chi_{q^2}(y)
=
q-1.
\]
Hence
\[
\sum_{\bfz \in \bA^{2g}(\bF_q)}
\chi_{q^2}\bigl(f_{\bfz}(x)\bigr)
=
q^{2g-1}(q-1).
\]

Combining the two cases, only the \(q\) elements of \(\bF_q\) contribute, so
\[
-\sum_{x \in \bF_{q^2}}\sum_{\bfz \in \bA^{2g}(\bF_q) } \chi_{q^2}(f_{\bfz}(x))
=
-\,q\cdot q^{2g-1}(q-1)
=
-\,q^{2g}(q-1).
\]

It remains to estimate the singular contribution
\[
S_{\mathrm{sing}}
:=
\sum_{\substack{\bfz \in \cP(w_g)(\bF_q)\\ \Delta(\bfz)\equiv 0}}
\frac{q-1}{\#\mu_{d(\bfz)}(\bF_q)} \sum_{x \in \bF_{q^2}} \chi_{q^2}(f_{\bfz}(x)).
\]
By (\ref{eqn: cP Fq to A Fq}), we have
\[
|S_{\mathrm{sing}}|
\le
\sum_{\substack{\bfz \in \bA^{2g}(\bF_q)\\ \Delta(\bfz)=0}}
\left|\sum_{x \in \bF_{q^2}} \chi_{q^2}(f_{\bfz}(x)) \right|.
\]
We claim that if $\Delta(\bfz) = 0$, then
\begin{align*}
    \sum_{x \in \bF_{q^2}} \chi_{q^2}(f_{\bfz}(x)) \ll_g q. 
\end{align*}
Indeed, we write $f_{\bfz}(x)=h_{\bfz}(x)^2 r_{\bfz}(x),$ where \(r_{\bfz}\) is squarefree. 
Since $\Delta(\bfz) = 0$, we have $h_{\bfz} \neq 1$.
For \(x\in \bF_{q^2}\) such that \(h_{\bfz}(x)\neq 0\), one has $\chi_{q^2}\bigl(f_{\bfz}(x)\bigr)=\chi_{q^2}\bigl(r_{\bfz}(x)\bigr),$ and
\[
-\sum_{x \in \bF_{q^2}} \chi_{q^2}(f_{\bfz}(x))
=
-\sum_{x\in \bF_{q^2}} \chi_{q^2}\bigl(r_{\bfz}(x)\bigr) + O_g(1).
\]
Then
\[
\left|
\sum_{x\in \bF_{q^2}} \chi_{q^2}\bigl(r_{\bfz}(x)\bigr)
\right|
\leq \left| \# C_{r_{\bfz}}(\bF_{q^2}) - (q^2+1) \right| \leq 2g(C_{r_{\bfz}}) q \ll_g q.
\]
Since \(\Delta\) is a nonzero polynomial in \(2g\) variables of degree depending only on \(g\), we have
\[
\#\{\bfz \in \bA^{2g}(\bF_q): \Delta(\bfz)=0\}\ll_g q^{2g-1}.
\]
Consequently, $S_{\mathrm{sing}} \ll_g q \cdot q^{2g-1}=O_g(q^{2g}).$
\end{proof}

\begin{lemma} \label{lem: cS aqsquare}
Let $q$ be an odd prime power. Then
\begin{align*}
    \sum_{\substack{\bfz \in \cP(w_g)(\bF_q) \\ F(f_{\bfz}) = (2, 1^{2g-1})}} \frac{q-1}{\# \mu_{d(\bfz)}(\bF_q)} a_{\fp^2}(C_{\bfz}) =O(q^{2g}).
\end{align*}
\end{lemma}
\begin{proof}
Since $a_{\fp^2}(C_{\bfz}) \ll q$, 
\begin{align*}
    \sum_{\substack{\bfz \in \cP(w_g)(\bF_q) \\ F(f_{\bfz}) = (2, 1^{2g-1}) }} \frac{q-1}{\# \mu_{d(\bfz)}(\bF_q)}  a_{\fp^2}(C_{\bfz})
    \ll q\sum_{\substack{\bfz \in \cP(w_g)(\bF_q) \\ F(f_{\bfz}) = (2, 1^{2g-1}) }} \frac{q-1}{\# \mu_{d(\bfz)}(\bF_q)}
    \ll q^{2g}
\end{align*}
by Proposition \ref{prop: affine partition bound}.
\end{proof}

\subsection{Proof of Theorem \ref{mainthm: rank}.}

Let $N = N_{K/\bQ}$ be the norm map.
In this section, we use $q_{\fp}$ for $\# \kappa(\fp)$ instead of $q$.

\begin{lemma} \label{lem: conductor bound}
Let $C/K$ be a hyperelliptic curve of genus $g$ with a Weierstrass point.
Then, the conductor $\ff(C)$ satisfies
\begin{align*}
    N(\ff(C)) \ll_{K, g} H_{w_g, K}(C)^{4g(2g+1)}.
\end{align*}
\end{lemma}
\begin{proof}
Let $f_{\fp}$ be the local conductor exponent of $C= C_{\bfa}$ at $\fp$ so that 
\begin{align*}
    N(\ff(C))=\prod_{\fp \in M_{K,0}} N(\fp)^{f_{\fp}(C)}.
\end{align*}
For a representative $\bfa$, we define
\begin{align*}
    m_{\fp}(\bfa) = \min_{2 \leq i \leq 2g+1} \left \lfloor \frac{\ord_{\fp}(a_i)}{2i} \right \rfloor, \qquad 
    \mathfrak{I}(\bfa) := \prod_{\fp} \fp^{m_{\fp}(\bfa)}, \qquad 
    \mathcal{D}(\bfa) := \Delta(f_{\bfa})\cO_{K} \cdot \mathfrak{I}(\bfa)^{-4g(2g+1)}.
\end{align*}
We note that the integral ideal $\mathcal{D}(\bfa)$ does not depend on the choice of $\bfa$.
Indeed, for $\bfa' = \lambda \cdot \bfa$ for $\lambda \in K^\times$, we have
\begin{align*}
    \mathfrak{I}(\bfa') = \lambda\cO_K \cdot  \mathfrak{I}(\bfa), \qquad 
    \Delta(f_{\bfa'}) = \lambda^{4g(2g+1)}\Delta(f_{\bfa})
\end{align*}
by \cite[\S 6]{CJP} and Lemma \ref{lem: disc well def} which shows that $\Delta$ is a weighted homogeneous polynomial of weighted degree $4g(2g+1)$. 
By the definition of the height, we have
\begin{align*}
    H_{w_g,K}(C)
    = \frac{1}{N(\mathfrak{I}(\bfa))}
    \prod_{v \in M_{K,\infty}}
    \max_{2 \leq i \leq 2g+1} |a_i|_v^{\frac{1}{2i}}.
\end{align*}
Hence
\begin{align*}
    N(\mathcal{D}(\bfa))
    &= \frac{|N_{K/\bQ}(\Delta(f_{\bfa}))|}
    {N(\mathfrak{I}(\bfa))^{4g(2g+1)}}
    \ll_{K,g} H_{w_g,K}(C)^{4g(2g+1)}.
\end{align*}

For each prime $\fp$, let $\pi$ be a uniformizer and put $\bfa_{\fp}=\pi^{-m_{\fp}(\bfa)}\cdot\bfa$, which is a weighted primitive representative.
By the weighted homogeneity of $\Delta$, we have
\begin{align*}
    \ord_{\fp}(\mathcal{D}(\bfa))
    &= \ord_{\fp}(\Delta(f_{\bfa}))
       -4g(2g+1)m_{\fp}(\bfa) 
    = \ord_{\fp}(\Delta(f_{\bfa_{\fp}})).
\end{align*}
For a prime $\fp \nmid 2$, by the conductor--discriminant inequality for hyperelliptic curves in odd residue characteristic \cite{OS24}, we have
\begin{align*}
    f_{\fp}(C_{\bfa})
    &\leq \ord_{\fp}(\Delta_{\fp}^{\min})
    \leq \ord_{\fp}(\Delta(f_{\bfa_{\fp}}))
    = \ord_{\fp}(\mathcal{D}(\bfa)),
\end{align*}
where $\Delta_{\fp}^{\min}$ is the local minimal discriminant of $C$ at $\fp$.
At primes above $2$, there is an explicit upper bound of the conductor exponent depending on $K$, $\fp$, and $g$ due to Brumer--Kramer \cite{BK94}.
Consequently,
\begin{align*}
    N(\ff(C_{\bfa})) \ll_{K,g} N(\mathcal{D}(\bfa))
    \ll_{K,g} H_{w_g,K}(C)^{4g(2g+1)}.
\end{align*}
This proves the lemma.
\end{proof}

Let $C$ be a hyperelliptic curve of genus $g$ over a number field $K$.
We use the normalization of the $L$-function so that the central point is $s = \frac{1}{2}$.
For every prime $\fp$, write
\begin{align*}
    L_{\fp}(s, C/K)^{-1} := \prod_{j=1}^{d_{\fp}}(1 - \alpha_{\fp, j}(C)q_{\fp}^{-s}), \qquad d_{\fp} \leq 2g
\end{align*}
with $|\alpha_{\fp, j}(C)| \leq 1$. For $m \geq 1$, we define $\lambda_{\fp^m}$ and $\Lambda_{\fp^m}$ so that
\begin{align*}
    \lambda_{\fp^m}(C):=\sum_{j=1}^{d_{\fp}}\alpha_{\fp,j}(C)^m, \qquad 
    -\frac{L'}{L}(s, C/K)= \sum_{\fp}\sum_{m=1}^\infty \frac{\Lambda_{\fp^m}(C)}{q_\fp^{ms}}.
\end{align*}
For every prime $\fp$ and $m \geq 1$, put
\begin{align*}
    a_{\fp^m}(C) := q_{\fp}^{\frac{m}{2}} \lambda_{\fp^m}(C).
\end{align*}
When $C$ has good reduction at $\fp$, then $d_{\fp} = 2g$, $|\alpha_{\fp, j}| = 1$ and $a_{\fp^m}(C) = q_{\fp}^m + 1 - \# C(\bF_{q_{\fp}^m})$.
More generally, Proposition \ref{prop: semistable simple double reduction} (3) shows that if every geometric root of $\overline{f}$ has multiplicity at most $2$, then $a_{\fp^m}(C) = q_{\fp}^m + 1 - \# \overline{C}(\bF_{q_{\fp}^m})$, where $\overline{C}$ denotes the corresponding naive special fiber.

Let $D_K$ be the discriminant of $K$ and 
\begin{align*}
    A(C) := |D_K|^{2g} N(\ff(C/K)).
\end{align*}
Then, Hasse--Weil conjecture for $C/K$ says that $L(s, C/K)$ has an analytic continuation, and the completed form
\begin{align*}
    \Lambda(s, C/K) := A(C)^{\frac{s + 1/2}{2}} \gamma(s, C) L(s, C/K)
\end{align*}
satisfies the conjectural functional equation
\begin{align*}
    \Lambda(s, C/K) = w(C) \Lambda(1-s, C/K)
\end{align*}
for the root number $w(C) \in \lcrc{ \pm 1}$ (cf. \cite[\S 1.3, (7)]{Ser70}). 

We recall an explicit formula for $L(s, C/K)$ (cf. \cite[Theorem 5.12]{IK}).

\begin{proposition} \label{prop: explicit}
Let $C$ be a hyperelliptic curve over a number field $K$.
Assume the Hasse--Weil conjecture and the generalized Riemann hypothesis for the $L$-function of $C$ over $K$.
Let $\phi$ be an even Schwartz function whose Fourier transform is compactly supported. Then, 
\begin{align*}
    \sum_{\rho} \phi\lbrb{\gamma\frac{\log X}{2\pi}}
    &=
    \widehat{\phi}(0)\frac{\log A(C)}{\log X} 
    -\frac{1}{\log X}
    \sum_{\fp}\sum_{m=1}^{\infty}
    \frac{2\Lambda_{\fp^m}(C)}{q_\fp^{m/2}}
    \widehat{\phi}\lbrb{\frac{m\log q_\fp}{\log X}} \\
    & +\frac{1}{\pi} \int_{-\infty}^{\infty}
    \mathrm{Re} \lbrb{ \frac{\gamma'}{\gamma}\lbrb{\frac12+it, C} }
    \phi\lbrb{\frac{t\log X}{2\pi}}\,dt,
\end{align*}
where $\rho=\frac12+i\gamma$ runs over the non-trivial zeros of $L(s, C/K)$.
\end{proposition}

Under the Hasse--Weil conjecture for hyperelliptic curves of genus $g$, for $C_{\bfa} \in \cH_{g, 1}(K)$, the $L$-function of $C_{\bfa}$ has an analytic continuation.
We denote its analytic rank by $r_{\bfa}$.
Let
\begin{align*}
    S_1 :=\frac{2}{\# \cH_{g, 1}(K)(X) \log X}  \sum_{ C_\bfa \in \cH_{g, 1}(K)(X) }
    \sum_{\fp} \frac{\lambda_{\fp}(C_{\bfa}) \log q_{\fp} }{\sqrt{q_{\fp}}} \widehat{\phi}\lbrb{\frac{\log q_{\fp}}{\log X}}
\end{align*}
and
\begin{align*}
    S_2 := \frac{2}{\# \cH_{g, 1}(K)(X) \log X}  \sum_{ C_\bfa \in \cH_{g, 1}(K)(X) }
    \sum_{\fp}\frac{\lambda_{\fp^2}(C_{\bfa}) \log q_{\fp} }{q_{\fp}} \widehat{\phi}\lbrb{\frac{2\log q_{\fp}}{\log X}}.
\end{align*}

\begin{proposition} \label{prop: average primitive}
Assume the Hasse--Weil conjecture and the generalized Riemann hypothesis for the $L$-functions of the hyperelliptic curves of genus $g$ with a Weierstrass point. 
Let $\phi$ be a nonnegative even Schwartz function whose Fourier transform $\widehat{\phi}$ is compactly supported and such that $\phi(0) > 0$.
Then,
\begin{align} \label{eqn: average rank bound}
    \frac{1}{\# \cH_{g, 1}(K)(X)} \sum_{ C_\bfa \in \cH_{g, 1}(K)(X)} r_{\bfa} \leq 4g(2g+1)\frac{ \widehat{\phi}(0)}{\phi(0)} - \frac{1}{\phi(0)}S_1 - \frac{1}{\phi(0)}S_2 + O_{K, g, \phi}\lbrb{\frac{1}{\log X}}.
\end{align}
\end{proposition}
\begin{proof}
Since the test function is nonnegative, we have
\begin{align*}
    \sum_{ C_\bfa \in \cH_{g, 1}(K)(X)} r_{\bfa}
    \leq \frac{1}{ \phi(0)} \sum_{ C_\bfa \in \cH_{g, 1}(K)(X)} \sum_{\rho_{\bfa}} \phi  \lbrb{\gamma_{\bfa}\frac{\log X}{ 2 \pi}}
\end{align*}
where $\rho_{\bfa} = \frac{1}{2} + i \gamma_{\bfa}$ are nontrivial zeros of $L(s, C_{\bfa}/K)$, under the generalized Riemann hypothesis.
We will apply the explicit formula from Proposition \ref{prop: explicit}, but note that the argument of \cite[Proposition 4.6]{JP26} shows that the contribution of gamma factors, and the sum of $\Lambda_{\fp^m}(C_{\bfa})$ with $m \geq 3$ are $O\lbrb{\frac{1}{\log X}}$.
Hence, together with Lemma \ref{lem: conductor bound}, we have
\begin{align*}
    &\frac{1}{ \# \cH_{g, 1}(K)(X)} \sum_{ C_\bfa \in \cH_{g, 1}(K)(X)} r_{\bfa}\\
    &\leq \frac{1}{ \# \cH_{g, 1}(K)(X)}  \frac{1}{\phi(0) \log X} \sum_{C_\bfa \in \cH_{g, 1}(K)(X)}
    \lbrb{  \widehat{\phi}(0) \log ( |D_K|^{2g} N(\ff(C_{\bfa}))   -  \sum_{\substack{\fp^k \\ k = 1, 2}} \frac{2\Lambda_{\fp^k}(C_{\bfa})  }{\sqrt{N(\fp^k)}}
    \widehat{\phi} \lbrb{\frac{\log N(\fp^k) }{ \log X}}} 
    \\
    & \qquad +O\lbrb{\frac{1}{\log X}} \\
    &\leq 4g(2g+1) \frac{\widehat{\phi}(0)}{\phi(0)} - \frac{1}{ \# \cH_{g, 1}(K)(X)}  \frac{1}{\phi(0)} \frac{1}{\log X} \sum_{C_\bfa \in \cH_{g, 1}(K)(X)}
    \sum_{\substack{\fp}} \frac{2\Lambda_{\fp}(C_{\bfa})  }{\sqrt{q_{\fp}}}
    \widehat{\phi} \lbrb{\frac{\log q_{\fp}}{ \log X}} \\
    & \qquad - \frac{1}{ \# \cH_{g, 1}(K)(X)}  \frac{1}{\phi(0)} \frac{1}{\log X} \sum_{C_\bfa \in \cH_{g, 1}(K)(X)}
    \sum_{\substack{\fp}} \frac{2\Lambda_{\fp^2}(C_{\bfa})  }{q_{\fp}}
    \widehat{\phi} \lbrb{\frac{2\log q_{\fp}}{ \log X}} +O\lbrb{\frac{1}{\log X}}.
\end{align*}
Plugging into the definitions of $S_i$, we obtain the result.
\end{proof}

\begin{lemma} \label{lem: S1 ap}
Let $\fp$ be a prime of $K$ of residue characteristic $p>2g+1$.
Then we have
    \begin{align*}
        \sum_{ \substack{C_\bfa \in \cH_{g, 1}(K)(X)  }  } a_{\fp}(C_{\bfa}) \ll 
        q_{\fp}^{\frac{3}{2}} X^{|w_g| - \frac{4}{D}} \log X + q_{\fp}^{-\frac{3}{2}}X^{|w_g|}.
    \end{align*}
\end{lemma}
\begin{proof}
Under the identification $\cH_{g, 1} \cong \cP(w_g) \setminus V(\Delta)$, we may rewrite the sum over $C_{\bfa} \in \cH_{g, 1}(K)(X)$ as the sum over $\bfa \in \cP(w_g)(K)(X)$ satisfying $\Delta(\bfa) \neq 0$.
By Lemma \ref{lem: Hg1 counting}, we have
\begin{align} \label{eqn: singular locus}
    \sum_{\substack{\bfa \in \cP(w_g)(K)(X) \\ \Delta(\bfa) = 0}} 1 = O\left(X^{|w_g| - (4g+2)}\right).
\end{align}
We have
\begin{align*}
    &\sum_{ \substack{C_\bfa \in \cH_{g, 1}(K)(X)  }  } a_{\fp}(C_{\bfa}) \\
    &= \sum_{\substack{C_\bfa \in \cH_{g, 1}(K)(X) \\ \Delta(\bfa) \not\equiv 0 \, (\fp) }} a_{\fp}(C_{\bfa})
    + \sum_{\substack{C_\bfa \in \cH_{g, 1}(K)(X) \\ F_{\fp}(f_{\bfa}) = (2, 1^{2g-1}) }} a_{\fp}(C_{\bfa})
    + \sum_{\substack{C_\bfa \in \cH_{g, 1}(K)(X)\\ r(F_{\fp}(f_{\bfa})) \leq 2g-1 \\ \psi_{\fp}(\bfa) \neq [0] }} a_{\fp}(C_{\bfa})
    + \sum_{\substack{C_\bfa \in \cH_{g, 1}(K)(X)\\ \psi_{\fp}(\bfa) = [0] }} a_{\fp}(C_{\bfa}).
\end{align*}
We note that $a_{\fp}(C_{\bfa}) = a_{\fp}(C_{\bfz})$ if $\Delta(\bfz) \not\equiv 0$, or if $F_{\fp}(f_{\bfa}) = (2, 1^{2g-1})$ by Proposition \ref{prop: semistable simple double reduction} (3).
By Proposition \ref{prop: counting P loc cond}, the first sum is
\begin{align*}
    &
    \sum_{\substack{\bfz \in \cP(w_g)(\bF_{q_{\fp}}) \\ \Delta(\bfz) \not\equiv 0 }}
    \sum_{\substack{\bfa \in \cP(w_g)(K)(X) \\ \psi_{\fp}(\bfa) = \bfz }} a_{\fp}(C_{\bfa}) \\
    &= 
    \sum_{\substack{\bfz \in \cP(w_g)(\bF_{q_{\fp}}) \\ \Delta(\bfz) \not\equiv 0   }} a_{\fp}(C_{\bfz})\lbrb{\frac{q_{\fp}-1}{\# \mu_{d(\bfz)}(\bF_{q_{\fp}})} \frac{1}{q_{\fp}^{2g}} \frac{1}{1 - q_{\fp}^{-|w_g|}} \kappa X^{|w_g|}  +  O\lbrb{q_{\fp}^{2-2g}X^{|w_g| - \frac{4}{D} } \log X} }.
\end{align*}
By Proposition \ref{prop: firstmoment zero}, we have
\begin{align*}
    \sum_{\substack{\bfz \in \cP(w_g)(\bF_{q_{\fp}}) \\ \Delta(\bfz) \not\equiv 0  }} \frac{q_{\fp}-1}{\# \mu_{d(\bfz)}(\bF_{q_{\fp}})} a_{\fp}(C_{\bfz}) = 0.
\end{align*}
So the main term vanishes.
Due to Weil bound $|a_{\fp}(C_{\bfa})| \leq 2g\sqrt{q_{\fp}}$ and (\ref{eqn: cP wg Fq counting}), the error term induces
\begin{align*}
    \ll q_{\fp}^{2 - 2g + 2g - 1 + \frac{1}{2}} X^{|w_g| - \frac{4}{D}} \log X
    = q_{\fp}^{\frac{3}{2}} X^{|w_g| - \frac{4}{D}} \log X.
\end{align*}
By \eqref{eqn: singular locus}, we have
\begin{align*}
    \sum_{\substack{C_\bfa \in \cH_{g, 1}(K)(X) \\ F_{\fp}(f_{\bfa}) = (2, 1^{2g-1}) }} a_{\fp}(C_{\bfa})
    = \sum_{\substack{\bfz \in \cP(w_g)(\bF_q) \\ F(f_\bfz) =(2, 1^{2g-1})}} a_{\fp}(C_{\bfz}) \# \lcrc{\bfa \in \cP(w_g)(K)(X) : \psi_{\fp}(\bfa) = \bfz} + O\left(\sqrt{q_{\fp}} X^{|w_g|-(4g+2)} \right).
\end{align*}
The main term vanishes due to Proposition \ref{prop: t=1 weighted vanishes}.
Its error term is of smaller order than the error term coming from the first sum, and hence may be absorbed into it.

For the third sum, Proposition \ref{prop: affine partition bound} shows that the relevant strata contain $O_g(q_{\fp}^{2g-3})$ weighted projective classes.
Using Proposition \ref{prop: counting P loc cond} and the bound $|a_{\fp}(C_{\bfa})|\leq 2g\sqrt{q_{\fp}}$, we obtain
\begin{align*}
& 
\sum_{\substack{C_{\bfa}\in\cH_{g,1}(K)(X)\\ r(F_{\fp}(f_{\bfa}))\leq 2g-1 }} a_{\fp}(C_{\bfa})
\ll_{K,g}
\sqrt{q_{\fp}} \left( q_{\fp}^{-2}X^{|w_g|} +q_{\fp}^{-1}X^{|w_g| - \frac{4}{D}}\log X \right) 
\ll_{K,g}
q_{\fp}^{-\frac{3}{2}}X^{|w_g|} +q_{\fp}^{-\frac{1}{2}}X^{|w_g| - \frac{4}{D}}\log X.
\end{align*}
Finally, by Proposition \ref{prop: counting P loc cond}, the contribution from the exceptional fiber $[0]$ is bounded by
\begin{align*}
\sum_{\substack{
C_{\bfa}\in\cH_{g,1}(K)(X)\\
\psi_{\fp}(\bfa)=[0]
}}
a_{\fp}(C_{\bfa})
\ll_{K,g}
\sqrt{q_{\fp}}\left( q_{\fp}^{-2g}X^{|w_g|} +q_{\fp}X^{|w_g| - \frac{4}{D}}\log X \right)
\ll_{K,g}
q_{\fp}^{-\frac{3}{2}}X^{|w_g|} +q_{\fp}^{\frac{3}{2}}X^{|w_g| - \frac{4}{D}}\log X.
\end{align*}
Combining the estimates for all four sums, we obtain the result.
\end{proof}

\begin{corollary} \label{cor: S1}
Let $\phi$ be a Schwartz even function whose support of the Fourier transform is in $[-\sigma, \sigma]$ for $\sigma >0$. Then, 
\begin{align*}
    S_1 \ll X^{\frac{3\sigma}{2} - \frac{4}{D}} + \frac{1}{\log X}.
\end{align*}
\end{corollary}
\begin{proof}
We note that the contribution of finitely many exceptional primes is negligible.
By Lemma \ref{lem: Hg1 counting} and Lemma \ref{lem: S1 ap},
\begin{align*}
    &S_1 = \frac{2}{\# \cH_{g, 1}(K)(X) \log X}  \sum_{ C_\bfa \in \cH_{g, 1}(K)(X) }
    \sum_{\fp} \frac{\lambda_{\fp}(C_{\bfa}) \log q_{\fp} }{\sqrt{q_{\fp}}} \widehat{\phi}\lbrb{\frac{\log q_{\fp}}{\log X}} \\
    &\ll \frac{1}{X^{|w_g|}\log X} \sum_{N(\fp) \leq X^\sigma} \frac{\log q_{\fp}}{q_{\fp}} \left| \sum_{ C_\bfa \in \cH_{g, 1}(K)(X) } a_{\fp}(C_\bfa) \right| \\
    & \ll \frac{1}{X^{|w_g|}\log X} \sum_{N(\fp) \leq X^\sigma} \frac{\log q_{\fp}}{q_{\fp}}\lbrb{q_{\fp}^{\frac{3}{2}} X^{|w_g| - \frac{4}{D}} \log X + q_{\fp}^{-\frac{3}{2}}X^{|w_g|}}.
\end{align*}
Since 
\begin{align*}
    \sum_{N(\fp) \leq X^\sigma} \sqrt{q_{\fp}} \log q_{\fp} \ll X^{\frac{3\sigma}{2}}, \qquad 
    \sum_{N(\fp) \leq X^\sigma} \frac{\log q_{\fp}}{q_{\fp}^{\frac{5}{2}}} \ll 1,
\end{align*}
we have the result.
\end{proof}

\begin{lemma} \label{lem: S2 ap}
Let $\fp$ be a prime of $K$ of residue characteristic $p > 2g+1$.
Then, we have
\begin{align*}
    \sum_{C_\bfa \in \cH_{g, 1}(K)(X)} a_{\fp^2}(C_{\bfa})
    = -q_{\fp} \kappa X^{|w_g|} + O_{K,g}\lbrb{X^{|w_g|} + q_{\fp}^2 X^{|w_g|-\frac{4}{D}}\log X}.
\end{align*}
\end{lemma}

\begin{proof}
By the same argument as in Lemma \ref{lem: S1 ap}, we have
\begin{align*}
    &\sum_{C_\bfa \in \cH_{g, 1}(K)(X)} a_{\fp^2}(C_{\bfa})\\
    &= \sum_{\substack{C_\bfa \in \cH_{g, 1}(K)(X) \\ \Delta(\bfa) \not\equiv 0 \, (\fp) }} a_{\fp^2}(C_\bfa)
    + \sum_{\substack{C_\bfa \in \cH_{g, 1}(K)(X) \\ F_{\fp}(f_{\bfa})=(2,1^{2g-1})}} a_{\fp^2}(C_\bfa)
     + \sum_{\substack{C_\bfa \in \cH_{g, 1}(K)(X) \\ r(F_{\fp}(f_{\bfa}))\leq 2g-1 \\ \psi_{\fp}(\bfa) \neq [0] }} a_{\fp^2}(C_{\bfa})
    + \sum_{\substack{C_\bfa \in \cH_{g, 1}(K)(X) \\ \psi_{\fp}(\bfa)=[0]}} a_{\fp^2}(C_{\bfa}).
\end{align*}
Again, $a_{\fp^2}(C_{\bfa})=a_{\fp^2}(C_{\bfz})$ if $\Delta(\bfz)\neq 0$, or if $F_{\fp}(f_{\bfa})=(2,1^{2g-1})$ by Proposition \ref{prop: semistable simple double reduction} (3).
For the first sum, by Proposition \ref{prop: counting P loc cond}, we have
\begin{align*}
    &\sum_{\substack{C_\bfa \in \cH_{g, 1}(K)(X) \\ \Delta(\bfa)\not\equiv 0 \, (\fp)}} a_{\fp^2}(C_\bfa)
    = \sum_{\substack{\bfz \in \cP(w_g)(\bF_{q_{\fp}}) \\ \Delta(\bfz)\neq 0}} a_{\fp^2}(C_{\bfz})
    \lbrb{\frac{q_{\fp}-1}{\#\mu_{d(\bfz)}(\bF_{q_{\fp}})}\frac{\kappa X^{|w_g|}}{q_{\fp}^{2g}(1-q_{\fp}^{-|w_g|})}
    +O\lbrb{q_{\fp}^{2-2g}X^{|w_g|-\frac{4}{D}}\log X}}.
\end{align*}
By Proposition \ref{prop: secondmoment}, the main term is
\begin{align*}
    \frac{\kappa X^{|w_g|}}{q_{\fp}^{2g}(1-q_{\fp}^{-|w_g|})}
    \sum_{\substack{\bfz \in \cP(w_g)(\bF_{q_{\fp}}) \\ \Delta(\bfz)\neq 0}}
    \frac{q_{\fp}-1}{\#\mu_{d(\bfz)}(\bF_{q_{\fp}})}a_{\fp^2}(C_{\bfz})
    &= \frac{\kappa X^{|w_g|}}{q_{\fp}^{2g}(1-q_{\fp}^{-|w_g|})} \lbrb{-q_{\fp}^{2g+1}+O(q_{\fp}^{2g})}\\
    &=-q_{\fp}\kappa X^{|w_g|}+O(X^{|w_g|}).
\end{align*}
Since $|a_{\fp^2}(C_{\bfz})|\ll_g q_{\fp}$ and $\#\cP(w_g)(\bF_{q_{\fp}})\ll_g q_{\fp}^{2g-1}$, the contribution of the error term is
\begin{align*}
    \ll_{K,g} q_{\fp}^{2-2g+2g-1+1}X^{|w_g|-\frac{4}{D}}\log X
    =q_{\fp}^2X^{|w_g|-\frac{4}{D}}\log X.
\end{align*}

For the second sum, we extend the summation to all points of $\cP(w_g)(K)(X)$ satisfying the same reduction condition. By \eqref{eqn: singular locus}, the difference is $O_g(q_{\fp}X^{|w_g|-(4g+2)})$. Hence, by Proposition \ref{prop: counting P loc cond}, we have
\begin{align*}
    \sum_{\substack{C_\bfa \in \cH_{g, 1}(K)(X) \\ F_{\fp}(f_{\bfa})=(2,1^{2g-1})}} a_{\fp^2}(C_\bfa)
    &=\frac{\kappa X^{|w_g|}}{q_{\fp}^{2g}(1-q_{\fp}^{-|w_g|})}
    \sum_{\substack{\bfz \in \cP(w_g)(\bF_{q_{\fp}}) \\ F(f_{\bfz})=(2,1^{2g-1})}}
    \frac{q_{\fp}-1}{\#\mu_{d(\bfz)}(\bF_{q_{\fp}})}a_{\fp^2}(C_{\bfz})\\
    &\qquad +O_{K,g}\lbrb{q_{\fp}X^{|w_g|-\frac{4}{D}}\log X+q_{\fp}X^{|w_g|-(4g+2)}}.
\end{align*}
The main term is $O_{K,g}(X^{|w_g|})$ by Lemma \ref{lem: cS aqsquare}, while the error term is absorbed into the claimed error.

For the third sum, Proposition \ref{prop: affine partition bound} shows that the relevant strata contain $O_g(q_{\fp}^{2g-3})$ weighted projective classes. By Proposition \ref{prop: counting P loc cond} and $|a_{\fp^2}(C_{\bfa})|\ll_g q_{\fp}$, we obtain
\begin{align*}
    &\sum_{\substack{C_\bfa \in \cH_{g, 1}(K)(X) \\ r(F_{\fp}(f_{\bfa}))\leq 2g-1}} a_{\fp^2}(C_{\bfa})
    \ll_{K,g}q_{\fp}\lbrb{q_{\fp}^{-2}X^{|w_g|}+q_{\fp}^{-1}X^{|w_g|-\frac{4}{D}}\log X}
    \ll_{K,g}q_{\fp}^{-1}X^{|w_g|}+X^{|w_g|-\frac{4}{D}}\log X.
\end{align*}

Finally, by Proposition \ref{prop: counting P loc cond}, the contribution of the exceptional fiber $[0]$ is bounded by
\begin{align*}
    &\sum_{\substack{C_\bfa \in \cH_{g, 1}(K)(X) \\ \psi_{\fp}(\bfa)=[0]}}a_{\fp^2}(C_{\bfa})
    \ll_{K,g}q_{\fp}\lbrb{q_{\fp}^{-2g}X^{|w_g|}+q_{\fp}X^{|w_g|-\frac{4}{D}}\log X}
    \ll_{K,g}X^{|w_g|}+q_{\fp}^2X^{|w_g|-\frac{4}{D}}\log X.
\end{align*}
Combining the estimates for all four sums, we obtain the result.
\end{proof}

\begin{corollary} \label{cor: S2}
Let $\phi$ be an even Schwartz function whose Fourier transform is supported in $[-\sigma,\sigma]$ for some $\sigma>0$, and let $D=[K:\bQ]$. Then,
\begin{align*}
    S_2=-\frac{\phi(0)}{2}+O_{K,g,\phi}\lbrb{\frac{1}{\log X}+X^{\frac{\sigma}{2}-\frac{4}{D}}}.
\end{align*}
\end{corollary}

\begin{proof}
The contribution of the finitely many exceptional primes is $O_{K,g,\phi}(1/\log X)$.
Since $\lambda_{\fp^2}(C)=q_{\fp}^{-1}a_{\fp^2}(C)$, by Lemmas \ref{lem: Hg1 counting} and \ref{lem: S2 ap}, we have
\begin{align*}
    S_2
    &=\frac{2}{\#\cH_{g,1}(K)(X)\log X}\sum_{\fp}\frac{\log q_{\fp}}{q_{\fp}^2}
    \widehat{\phi}\lbrb{\frac{2\log q_{\fp}}{\log X}}
    \sum_{C_\bfa\in\cH_{g,1}(K)(X)}a_{\fp^2}(C_{\bfa})\\
    &=-\frac{2\kappa X^{|w_g|}}{\#\cH_{g,1}(K)(X)\log X}
    \sum_{\fp}\frac{\log q_{\fp}}{q_{\fp}}
    \widehat{\phi}\lbrb{\frac{2\log q_{\fp}}{\log X}}
    +O_{K,g,\phi}\lbrb{\frac{1}{\log X}+X^{\frac{\sigma}{2}-\frac{4}{D}}}.
\end{align*}
Indeed, the contribution of the error term in Lemma \ref{lem: S2 ap} is
\begin{align*}
    &\ll_{K,g,\phi}\frac{1}{X^{|w_g|}\log X}
    \sum_{q_{\fp}\leq X^{\sigma/2}}\frac{\log q_{\fp}}{q_{\fp}^2}
    \lbrb{X^{|w_g|}+q_{\fp}^2X^{|w_g|-\frac{4}{D}}\log X}\ll_{K,g,\phi}\frac{1}{\log X}+X^{\frac{\sigma}{2}-\frac{4}{D}},
\end{align*}
since $\sum_{\fp}\log q_{\fp}/q_{\fp}^2\ll_K 1$ and $\sum_{q_{\fp}\leq X}\log q_{\fp}\ll_K X$.
By \cite[Lemma 5.5]{CJP}, we have
\begin{align*}
    \frac{2}{\log X}\sum_{\fp}\frac{\log q_{\fp}}{q_{\fp}}
    \widehat{\phi}\lbrb{\frac{2\log q_{\fp}}{\log X}}
    =\frac{\phi(0)}{2}+O_{K,\phi}\lbrb{\frac{1}{\log X}}.
\end{align*}
Also, Lemma \ref{lem: Hg1 counting} gives
\begin{align*}
    \frac{\kappa X^{|w_g|}}{\#\cH_{g,1}(K)(X)}
    =1+O_{K,g}\lbrb{X^{-\frac{4}{D}}\log X}.
\end{align*}
Hence, we obtain the result.
\end{proof}


\begin{theorem}
Assume the Hasse--Weil conjecture and the generalized Riemann hypothesis for the $L$-functions of hyperelliptic curves over $K$ of genus $g$ with a Weierstrass point.
For $D=[K:\bQ]$, we have
\begin{align*}
    \limsup_{X\to\infty}
    \frac{1}{\#\cH_{g,1}(K)(X)}\sum_{C_\bfa\in\cH_{g,1}(K)(X)}r_{\bfa}
    \leq \frac{1}{2}+\frac{3g(2g+1)D}{2}.
\end{align*}
\end{theorem}

\begin{proof}
Fix $0<\sigma<\frac{8}{3D}$, and let $\phi$ be a nonnegative even Schwartz function with $\phi(0)>0$ whose Fourier transform is supported in $[-\sigma,\sigma]$.
By Proposition \ref{prop: average primitive}, Corollary \ref{cor: S1}, and Corollary \ref{cor: S2}, we obtain
\begin{align*}
    \frac{1}{\#\cH_{g,1}(K)(X)}\sum_{C_\bfa\in\cH_{g,1}(K)(X)}r_{\bfa}
    \leq \frac12+4g(2g+1)\frac{\widehat{\phi}(0)}{\phi(0)}
    +O_{K,g,\phi}\lbrb{\frac{1}{\log X}+X^{\frac{3\sigma}{2}-\frac{4}{D}}}.
\end{align*}
Since $\sigma<\frac{8}{3D}$, the error term tends to zero as $X\to\infty$. Hence,
\begin{align*}
    \limsup_{X\to\infty}
    \frac{1}{\#\cH_{g,1}(K)(X)}\sum_{C_\bfa\in\cH_{g,1}(K)(X)}r_{\bfa}
    \leq \frac12+4g(2g+1)\frac{\widehat{\phi}(0)}{\phi(0)}.
\end{align*}

We now construct a family of admissible test functions approximating $\frac{\sin^2(\pi\sigma x)}{(2\pi x)^2}$.
For $\varepsilon>0$, define
\begin{align*}
    \psi_{\sigma,\varepsilon}(u):=
    \begin{cases}
        \exp\lbrb{-\frac{\varepsilon}{1-4u^2/\sigma^2}}, & |u|<\frac{\sigma}{2},\\
        0, & |u|\geq \frac{\sigma}{2},
    \end{cases}, \qquad
    \phi_{\sigma,\varepsilon}(x):=\frac14\left|\int_{\bR}\psi_{\sigma,\varepsilon}(u)e^{2\pi iux}\,du\right|^2.
\end{align*}

Since $\psi_{\sigma,\varepsilon}$ is real-valued and even, its inverse Fourier transform is also real-valued. Hence, by changing the order of integration, we have
\begin{align*}
    4\phi_{\sigma,\varepsilon}(x)
    &=\left(\int_{\bR}\psi_{\sigma,\varepsilon}(u)e^{2\pi iux}\,du\right)^2=\int_{\bR}\left(\int_{\bR}\psi_{\sigma,\varepsilon}(u)\psi_{\sigma,\varepsilon}(v-u)\,du\right)e^{2\pi ivx}\,dv.
\end{align*}
By Fourier inversion, we obtain
\begin{align*}
    \widehat{\phi}_{\sigma,\varepsilon}(v)
    =\frac14\int_{\bR}\psi_{\sigma,\varepsilon}(u)\psi_{\sigma,\varepsilon}(v-u)\,du.
\end{align*}
In particular, $\widehat{\phi}_{\sigma,\varepsilon}$ is supported in $[-\sigma,\sigma]$, and
\begin{align*}
    \widehat{\phi}_{\sigma,\varepsilon}(0)=\frac14\int_{\bR}\psi_{\sigma,\varepsilon}(u)^2\,du,\qquad
    \phi_{\sigma,\varepsilon}(0)=\frac14\left(\int_{\bR}\psi_{\sigma,\varepsilon}(u)\,du\right)^2.
\end{align*}
Therefore, by the change of variables $u=\sigma t/2$, we have
\begin{align*}
    \frac{\widehat{\phi}_{\sigma,\varepsilon}(0)}{\phi_{\sigma,\varepsilon}(0)}
    =\frac{1}{\sigma}\frac{\int_0^1 e^{-\frac{2\varepsilon}{1-t^2}}\,dt}
    {\left(\int_0^1 e^{-\frac{\varepsilon}{1-t^2}}\,dt\right)^2}
    \longrightarrow\frac{1}{\sigma}
\end{align*}
as $\varepsilon\to0$, by the dominated convergence theorem.
Applying the preceding inequality to $\phi_{\sigma,\varepsilon}$ and letting $\varepsilon\to0$, we obtain
\begin{align*}
    \limsup_{X\to\infty}\frac{1}{\#\cH_{g,1}(K)(X)}
    \sum_{C_\bfa\in\cH_{g,1}(K)(X)}r_{\bfa}
    \leq \frac12+\frac{4g(2g+1)}{\sigma}.
\end{align*}
Finally, letting $\sigma\to\frac{8}{3D}$ from below gives
\begin{align*}
    \limsup_{X\to\infty}\frac{1}{\#\cH_{g,1}(K)(X)}
    \sum_{C_\bfa\in\cH_{g,1}(K)(X)}r_{\bfa}
    \leq \frac12+\frac{3g(2g+1)D}{2}.
\end{align*}
This completes the proof.
\end{proof}



\begin{thebibliography}{alpha}

\bibitem[BBB+22]{BBB} A. J. Best, L. A. Betts, M. Bisatt, R. van Bommel, V. Dokchitser, O. Faraggi, S. Kunzweiler, C. Maistret, A. Morgan, S. Muselli, S. Nowell, A user's guide to the local arithmetic of hyperelliptic curves. Bull. Lond. Math. Soc. \textbf{54} (2022), no.3, 825–867.


\bibitem[BG13]{BG13} M. Bhargava, B. H. Gross, The average size of the 2-Selmer group of Jacobians of hyperelliptic curves having a rational Weierstrass point. Automorphic representations and L-functions, 23–91.
Tata Inst. Fundam. Res. Stud. Math., 22
Published for the Tata Institute of Fundamental Research, Mumbai; by, 2013


\bibitem[BK94]{BK94} A. Brumer, K. Kramer, The conductor of an abelian variety.
Compositio Math. \textbf{92} (1994), no. 2, 227–248.



\bibitem[BN22]{BN} P. Bruin, F. Najman, Counting elliptic curves with prescribed level structures over number fields. J. Lond. Math. Soc. (2) \textbf{105} (2022), no.4, 2415–2435.

\bibitem[BW17]{BW} I. I. Bouw, S. Wewers, Computing $L$-functions and semistable reduction of superelliptic curves.  Glasg. Math. J. 59 (2017), no. 1, 77–108.



\bibitem[CJ23a]{CJ1} P. J. Cho, K. Jeong, On the distribution of analytic ranks of elliptic curves. Math. Z. \textbf{305} (2023), no. 3, Paper No. 42, 20 pp.

\bibitem[CJ23b]{CJ2} P. J. Cho, K. Jeong, The average analytic rank of elliptic curves with prescribed torsion. J. Lond. Math. Soc. (2) \textbf{107} (2023), no.2, 616–657.

\bibitem[CJP25]{CJP} P. J. Cho, K. Jeong, J. Park, The average analytic rank of elliptic curves with prescribed level structure, preprint.



\bibitem[CY26]{CY} P. J. Cho, J. Yoo, On the ranks of $L$-functions of hyperelliptic curves in certain family, preprint.


\bibitem[DDMM]{DDMM} T. Dokchitser, V. Dokchitser, C. Maistret, A. Morgan, Arithmetic of hyperelliptic curves over local fields. Math. Ann. 385 (2023), no. 3-4, 1213–1322.



\bibitem[HP23]{HP23} C. Han, J-Y Park, Enumerating odd-degree hyperelliptic curves and abelian surfaces over $\mathbb{P}^1$. Math. Z. \textbf{304} (2023), no. 1, Paper No. 5, 32 pp.

\bibitem[IK04]{IK} H. Iwaniec, E. Kowalski, Analytic number theory. Amer. Math. Soc. Colloq. Publ., 53 American Mathematical Society, Providence, RI, 2004. xii+615 pp.

\bibitem[JRS24]{JRS} B. W. Jordan, K. A, Ribet, A. J. Scholl, Modular curves and Néron models of generalized Jacobians. Compos. Math. 160 (2024), no. 5, 945–981.



\bibitem[Liu02]{Liu02} Q. Liu, Algebraic geometry and Arithmetic Curves, Oxford University Press, 2002.


\bibitem[OS24]{OS24} A. Obus, P. Srinivasan,
Conductor-discriminant inequality for hyperelliptic curves in odd residue characteristic. Int. Math. Res. Not. IMRN 2024, no. 9, 7343–7359.

\bibitem[JP26]{JP26} K. Jeong, J. Park, Goldfeld conjecture for non-hyperelliptic direction, preprint.


\bibitem[Phi25]{Phi2} T. Phillips, Average analytic ranks of elliptic curves over number fields. Forum Math. Sigma \textbf{13} (2025), Paper No. e40, 36 pp.



\bibitem[Ser70]{Ser70} J.-P. Serre, Facteurs locaux des fonctions z\^{e}ta des variet\'es alg\'ebriques (d\'efinitions et conjectures). S\'eminaire Delange-Pisot-Poitou. 11e ann\'ee: 1969/70. Th\'eorie des nombres. Fasc. 1: Expos\'es 1 {\`a} 15; Fasc. 2: Expos\'es 16 {\`a} 24, 15 pp. Secr\'etariat Math\'ematique, Paris, 1970




\bibitem[Sta18]{Stacks}
Stacks Project Authors: \emph{The Stacks Project}, \url{https://stacks.math.columbia.edu} (2018)




\end{thebibliography}
\end{document}